\documentclass{article}
\usepackage{amsmath}
\usepackage{amssymb}
\usepackage{amsfonts}
\usepackage{amsthm}
\usepackage{graphicx}
\usepackage{geometry}
\usepackage{algorithm}
\usepackage{algorithmic}
\usepackage{enumitem}
\usepackage{float}
\title{Near Optimality of Discrete-Time Approximations for Controlled McKean-Vlasov and Large Interacting Particle Diffusions}
\author{Somnath~Pradhan\thanks{S. Pradhan is with the Department
of Mathematics, Indian Institute of Science Education and Research Bhopal, MP-462066, India e-mail: (somnath@iiserb.ac.in).}\,\,
and Serdar~Y\"{u}ksel\thanks{S. Y\"{u}ksel is with Department of Mathematics and Statistics,
Queen's University, Kingston, ON, Canada e-mail: (yuksel@queensu.ca).}}

\usepackage{amssymb,nicefrac,bm,upgreek,mathtools,verbatim}
\usepackage[final]{hyperref}
\usepackage[mathscr]{eucal}
\usepackage{dsfont}
\usepackage[normalem]{ulem}
\usepackage{amsopn}
\allowdisplaybreaks
\newtheorem{theorem}{Theorem}[section]
\newtheorem{lemma}[theorem]{Lemma}
\newtheorem{corollary}[theorem]{Corollary}

\newtheorem{remark}[theorem]{Remark}

\newtheorem{assumption}{Assumption}

\numberwithin{equation}{section}
\newcommand{\Uadm}{\mathfrak U}

\newcommand{\Act}{\mathbb{U}}
\newcommand{\Usm}{\mathfrak U_{\mathsf{sm}}}
\newcommand{\Udsm}{\mathfrak U_{\mathsf{dsm}}}

\newcommand{\Um}{\mathfrak U_{\mathsf{m}}}
\newcommand{\pV}{\mathrm{V}} % Probability measures space
\newcommand{\fB}{{\mathfrak{B}}}  % Borel Sets
\newcommand{\cC}{{\mathcal{C}}}   % Continuous functions
\newcommand{\sD}{{\mathscr{D}}}   % Skorokhod space 
\newcommand{\bE}{{\mathbb{E}}} 
\newcommand{\sF}{{\mathfrak{F}}}   % sigma field

\newcommand{\cJ}{{\mathcal{J}}}  % discounted cost
  
\newcommand{\sL}{{\mathscr{L}}}  %  Generator of Diffusion
\newcommand{\cM}{{\mathcal{M}}} % policies relaxed
\newcommand{\fM}{{\mathfrak{M}}} % for policies relaxed

\newcommand{\cP}{{\mathcal{P}}}  % Probability measures
\newcommand{\sP}{{\mathscr{P}}}
\newcommand{\RR}{\mathds{R}}
\newcommand{\NN}{\mathds{N}}
\newcommand{\ZZ}{\mathds{Z}}

\newcommand{\Rd}{{\mathds{R}^{d}}}
\DeclareMathOperator{\Exp}{\mathbb{E}}

\newcommand{\D}{\mathrm{d}}

\newcommand{\df}{:=}
\newcommand{\transp}{^{\mathsf{T}}}

\newcommand{\order}{{\mathscr{O}}}

\usepackage{color}

\newcommand{\abs}[1]{\lvert#1\rvert}
\newcommand{\norm}[1]{\lVert#1\rVert}
\usepackage{tikz}
\usetikzlibrary{positioning,arrows}
\begin{document}
\date{}
\maketitle

\begin{abstract}
We study stochastic optimal control problems for (possibly degenerate) McKean-Vlasov controlled diffusions and obtain discrete-time as well as finite interacting particle approximations. (i) Via continuity of the expected cost in control policy by endowing the space of relaxed policies with a compact weak topology, we prove near-optimality of piecewise-constant policies which leads to a discrete-time model. We show that the discrete-time value functions (for finite-horizon and discounted infinite-horizon) converge to their continuous-time counterparts as the timestep converges to zero. In particular, we establish that optimal policies for the discrete-time model exists and they are near-optimal for the original continuous-time problem. (ii) We then extend these approximation and near-optimality results to $N$-particle interacting systems under centralized or decentralized mean-field sharing information structure, proving that the discrete-time McKean-Vlasov policy is asymptotically optimal as $N\to \infty$ and the time step goes to zero. Using discrete-time approximation as an intermediate step leads to complementary conditions compared to those in the literature. (iii) We thus develop a unified approximation framework for McKean-Vlasov optimal control problems via discrete-time McKean-Vlasov control problems (and associated numerical methods such as finite model approximations), and we also show near optimality of such approximate policy solutions for the $N$-agent interacting models under centralized and decentralized control.
\end{abstract}

% \begin{keyword}[class=MSC]
% \kwd[Primary ]{93E20}
% \kwd{49K45}
% \kwd[; secondary ]{60K35}
% \kwd{60H30}
% \end{keyword}

% \begin{keyword}
% \kwd{Controlled McKean–Vlasov diffusion}
% \kwd{Discrete-time approximation}
% \kwd{Euler-Maruyama scheme}
% \kwd{Error estimate}
% \kwd{Interacting particle systems}
% \end{keyword}
%%%%%%%%%%%%%%

%%%%%%%%%%%%%%%%%%%%%%%%%%%%%%%%%%%%%%%%%%%%%%
%% Please use \tableofcontents for articles %%
%% with 50 pages and more                   %%
%%%%%%%%%%%%%%%%%%%%%%%%%%%%%%%%%%%%%%%%%%%%%%
\tableofcontents

%%%%%%%%%%%%%%%%%%%%%%%%%%%%%%%%%%%%%%%%%%%%%%
%%%% Main text entry area:
\section{Introduction}

Stochastic optimal control theory is a fundamental area in applied mathematics and engineering, concerned with optimizing the behavior of systems influenced by random noise. In classical stochastic control, the dynamics are typically described by stochastic differential equations (SDEs) where the drift and diffusion coefficients depend on the state and control. However, many complex systems, particularly those arising in engineering, economics, finance, physics, and social sciences, exhibit interactions between a large number of agents. These interactions often manifest in a mean-field manner, where the behavior of an individual agent is influenced by the aggregate behavior, or the distribution, of all agents. This leads to the concept of McKean-Vlasov systems, where the dynamics of each agent depend not only on its own state and control but also on the probability distribution of the system's state. The dynamics for such systems is typically given by an SDE of the type
\begin{align}\label{MVModel1}
dX_t = b(X_t, \sL(X_t), U_t) dt + \upsigma(X_t, \sL(X_t)) dW_t, \quad X_0 = x \in \mathbb{R}^d
\end{align}
where $\sL(X_t)$ denotes the law (probability distribution) of $X_t$, $U_t$ is a control process and $W_t$ is a standard Brownian motion. More detailed discussion on this equation and conditions on the drift and diffusion terms above are presented further below. 

In this paper, we develop a unified approximation and regularity framework involving the continuous-time model (\ref{MVModel1}), its discrete-time approximation (see \ref{Eeuler_maruyama}), a continuous-time $N$-interacting particle model (see \ref{ENState}), and a discrete-time $N$-interacting particle model (see \ref{ENStateDiscrete}). Precise models and formal results are presented further below, and an overview of our contributions is given in Figure \ref{fig:approximation_roadmap}.

Stochastic optimal control for such McKean-Vlasov diffusions represents a significant extension of classical stochastic control theory. It addresses scenarios where the control objective is to optimize a system described by a McKean-Vlasov SDE. This framework has gained increasing attention due to its ability to model large population systems, mean-field games, and other complex phenomena exhibiting distributional dependence. In particular, these processes are valuable in modeling various real-world phenomena, e.g. (i) {\it Collective dynamics of interacting agents and robotics:} Coordinating the movements of a large number of agents or robots, where each agent adjusts its behavior based on the positions of its neighbors or the empirical distribution of the other agents; see e.g.  \cite{carrillo2022controlling,albi2022moment,rapakoulias2025steering,toumi2024mean,lee2025traffic,liu2024optimal}. (ii) {\it Generative learning, distributional approximations, and transformers:} Machine learning models increasingly utilize interactive agent dynamics to both map input ensembles to output ensembles as well as to generate stochastic data with prespecified marginal or joint distributions; see e.g. \cite{geshkovski2025mathematical,awadelkarim2025particle,chen2025quantitative,geshkovski2024measure}. Furthermore, the score-based diffusion method for generative modeling \cite{song2020score} critically builds on the study of backwards stochastic differential equations which are of the McKean-Vlasov type \cite{anderson1982reverse}. Optimal transport over dynamical systems (see e.g. \cite{chen2021stochastic,chen2021optimal}) has emerged as an ideal formulation for this problem \cite{alouadi2025robust}. (iii) {\it Financial mathematics:} Capturing collective behavior among investors, where trading decisions are influenced by the overall market sentiment; see e.g. \cite{borkar2010mckean,pham2016linear}. (iv) {\it Epidemiology:} Modeling the spread of infectious diseases under control, accounting for the influence of vaccination strategies and social distancing measures which serve as control policies; see e.g. \cite{morsky2023impact}.

The analysis of these control problems presents unique challenges compared to classical stochastic control. The dependence on the distribution introduces further non-linearity in the system. The presence of the probability distribution in the dynamics transforms the control problem into one defined on an infinite-dimensional space, requiring sophisticated analytical tools. The standard dynamic programming approach, which relies on solving the Hamilton-Jacobi-Bellman (HJB) equation, becomes significantly more complex in this setting due to the infinite-dimensional nature of the state space, and the non-Markovian property of the (individual) agent dynamics.

Significant progress has been made in the development of an optimal control theory for McKean-Vlasov SDEs. The existence of optimal control policies has been established under various conditions, often relying on the direct method of calculus of variations or on the existence of solutions to the associated mean-field HJB equation. The optimal control of such problems (including under game theoretic formulations) has been studied in \cite{pham2016linear,carmona2018probabilistic,bayraktar2018randomized,pham2017dynamic,achdou2015system,lauriere2014dynamic,hofer2024potential} under complementary formulations and assumptions. Classical dynamic programming methods give rise to Hamilton–Jacobi–Bellman (HJB) equations over the space of measures. Recent work has rigorously developed this framework: for example, \cite{lauriere2014dynamic,pham2017dynamic} proved a dynamic programming principle for such control problems and derived the corresponding HJB equation on the space of probability measures, though with some restrictions such as Lipschitz regularity in \cite{pham2017dynamic} and more specific formulation in \cite{lauriere2014dynamic}; these conditions were then relaxed in \cite{bayraktar2018randomized}. \cite{lacker2017limit,lacker2015mean} utilized weak convergence as well as Bellman optimality methods to arrive at the existence of optimal solutions for possibly degenerate diffusions of McKean-Vlasov type. Further recent references include \cite{djete2022mckean,cardaliaguet2023algebraic,daudin2024optimal}. While these analytical results characterize the value function, solving the resulting infinite-dimensional HJB equation explicitly is generally intractable. %Continuity-compactness arguments are also used to demonstrate the existence of minimizing sequences and the convergence of these sequences to an optimal control.
%%%%%%%%%%%%%%%%%%%%%%

%\sy{Also: Pham has a recent paper on Diffusion modeling %via Vlasov analysis. Cite the paper of Curtis-Barron %also.}

There are two main solution approaches for controlled McKean–Vlasov problems.  The \emph{probabilistic} (weak / relaxed / maximum-principle) approach works directly with stochastic formulations: it permits open-loop or measure-valued (relaxed) controls, builds compactness/tightness in the space of laws, and proves existence optimal/ near-optimal policies by weak convergence, martingale problem formulation, and chattering arguments \cite{CarmonaDelarue18I,lacker2015mean,lacker2017limit,lacker2018probabilistic}. As we will show, this approach leads directly to implementable time discretizations under relatively mild Lipschitz continuity and growth hypotheses.

On the other hand, the \emph{PDE / dynamic-programming} (HJB / master-equation) approach assumes or arrives at the Markov/feedback structure of the control policies, so that one can formulate the value function in an appropriate function space. To formulate and analyze the HJB/master PDE, we often require higher-order regularity of the system model, uniform ellipticity or nondegeneracy/ convexity assumptions. Treating common noise in this framework typically forces conditional-law or stochastic-PDE formulations which demand even more structure \cite{CarmonaDelarue18I,DjetePossamaiTan2022,LionsLecture}.  However, the PDE/HJB route is appropriate when one seeks a sharp value-function characterization and wishes to construct an optimal feedback policy under strong regularity assumptions.

%This paper investigates the stochastic optimal control %problem for McKean-Vlasov controlled diffusions, focusing %on the development of \textit{discrete-time %approximations}, which then gives rise to the alternative %paradigms of {\it finite agent approximations} as well as %%{\it finite model approximations}, whose solutions are then %near optimal for the original McKean-Vlasov model. To this %end, the \textit{existence} of optimal control policies is %also studied as a critical supporting result via relatively %more relaxed conditions. 

A discrete-time mean-field theory leads to a formulation which is ideally suited for implementation of learning and approximation algorithms. The existence and dynamic programming for McKean-Vlasov dynamics was first studied in \cite{pham2016discrete} to our knowledge, and later both existence and approximations under a weak Feller continuity analysis was presented in \cite{sanjarisaldiy2024optimality}. Accordingly, our analysis invites a direct and accessible approach for the control of such processes. As formulated and analyzed in \cite[Section 7.2]{sanjarisaldiy2024optimality}, for the discrete-time McKean-Vlasov formulations, one can obtain finite model approximations with rigorous near optimality guarantees via measure valued state and action quantization building on the weak Feller regularity and the analysis in \cite{SaYuLi15c}. Accordingly, the analysis in our paper has significant implications for optimal control, approximations, and learning of continuous-time McKean-Vlasov dynamics and their control. 

We also note that for the control-free setup, existence of strong or weak solutions has been studied extensively (see e.g. \cite[Theorem~4.21]{carmona2018probabilistic} and \cite{mishura2020existence} for very general conditions relaxing Lipschitz regularity in the drift term for strong solutions as well as weak solutions), and for such control-free models discrete-time approximations have been studied in \cite{li2023strong} among others.

A corollary of our analysis is a contribution to the theory on the near optimality of McKean-Vlasov solutions for $N$-agent interacting particle diffusion models when $N$ is large under relatively weaker conditions than those available in the literature presented in Theorem \ref{nearOptDiscFiniteHor}: Under various technical conditions, the sequence of solutions of an $N$-agent interacting particle system with centralized information structure admits an optimal solution which converges, possibly via a subsequence, to the McKean-Vlasov solution as $N \to \infty$: see \cite{fornasier2014mean,fornasier2019mean} for a deterministic setup, and  \cite{lacker2017limit,jackson2023rate} for controlled interacting diffusions, under different assumptions and solution approaches. An analogous result is obtained in discrete-time using the theory of $N$-exchangeable processes and their limits to exchangeable processes \cite{sanjarisaldiy2024optimality} (see also \cite{bauerle2023mean} as a critical study used for an intermediate analysis). While most of these papers focused on asymptotic convergence, \cite{jackson2023rate} established a  quantitative \(1/N\) convergence rate under a more restrictive model. %Their analysis relies on smoothness properties of the value function via PDE techniques (for the additive noise case). 

%Notably, we provide sufficient conditions which ensure the existence of an %optimal policy of the continuous-time McKean-Vlasov control problem. Our %approach is closely related to the continuity and compactness methods used %in the mean-field control literature, particularly the %works~\cite{lacker2017limit,lacker2015mean}, where the analysis is carried %out in the framework of controlled martingale problems. More generally, the %literature~\cite{pham2017dynamic,lauriere2014dynamic,bayraktar2018randomized,%lacker2017limit,lacker2015mean} develops weak convergence, compactification, %and dynamic programming methods for McKean--Vlasov control problems in %fairly general settings, including models with unbounded coefficients and %noncompact control spaces. 

The main contributions of this paper are the following:
\begin{itemize}
\item While existence of optimal policies is not a new contribution, as a general analysis is already reported in the literature notably by Lacker \cite{lacker2017limit,lacker2015mean} also for models with unbounded coefficients and noncompact control spaces, we utilize the continuous dependence of solution measures in control policy under a suitable metric as a critical step towards arriving at near optimal solutions. In particular, rather than a martingale characterization of solutions, we show the $L_2$ continuity of solutions in control policy, which allows for a quantitative convergence analysis: Building on such a continuity (a corollary of which is existence) analysis, as a first primary contribution, we show that a discrete-time approximation (given in (\ref{SdiscreteSDE})) can be obtained so that the value of the discrete-time model converges to that of the continuous-time problem. We work under boundedness and Lipschitz regularity assumptions, which allow us to obtain explicit quantitative approximation estimates; see Corollary~\ref{cor:EMAppro1AMV} and Corollary~\ref{cor:CNearFinite}. Furthermore, a piecewise constant extension of the optimal control policy of the discrete-time model is shown to be near-optimal for the continuous-time model (see Theorems \ref{TNearFinite} and \ref{TNeardiscounted}). 

In this direction, we also establish the existence of an optimal solution to the discrete-time approximate model in Theorem \ref{TOptiDiscMcKV1A} by endowing the discrete-time control policy with a suitable Young topology under which a weak Feller condition holds for the (sampled) discrete-time model. 

To our knowledge, except for the recent contemporary study \cite{reisinger2025convergence} such an analysis on discrete-time approximations is largely absent from the literature. The paper \cite{reisinger2025convergence} establishes stronger quantitative approximation results, including explicit convergence rates for both the value function and the associated control policies, under additional structural and regularity assumptions. In particular, the analysis relies on conditions such as linear dependence of the drift on the control variable and higher regularity of the value function, notably the assumption $(H.5)$, which requires 
\[
V_\pi^c \in C^{1,2}_2\big([0,T]\times \mathcal{P}_2(\mathbb{R}^n)\big).
\]
These assumptions enable sharper quantitative estimates and approximation results. In the present work, we instead adopt a probabilistic framework which applies under comparatively milder assumptions (boundedness and Lipschitz continuity), while still yielding quantitative approximation estimates for the value functional and near-optimality properties for the associated discrete-time optimal control policies.

% \item Building on the existence analysis, as a first primary result, we show that a discrete-time approximation (given in (\ref{SdiscreteSDE})) can be obtained so that the value of the discrete-time approximation converges to the value of the continuous-time model and more operationally, a piece-wise constant extension of the optimal control policy of the discrete-time model is near optimal for the continuous-time model (see Theorems \ref{TNearFinite} and \ref{TNeardiscounted}). To this end, we also show that an optimal solution to the discrete-time approximate model exists. To our knowledge, except for the recent contemporary study \cite{reisinger2025convergence} such an analysis on discrete-time approximations is not available in the literature: Imposing more demanding conditions on the model such as linear dependence of the drift on the control, \cite{reisinger2025convergence} obtains rates of convergence (including on policies). Our approach, utilizing weak convergence, allows for significantly more relaxed conditions.

\item An implication of our analysis is on the near optimality of McKean-Vlasov solutions for $N$-agent interacting particle diffusion models when $N$ is large under complementary conditions to those present in the literature \cite{lacker2017limit,jackson2023rate}: Theorems \ref{nearOptDiscFiniteHor} and \ref{nearOptDiscDiscount} establish near optimality of discrete-time policies of the type $U(t,X_t,\mu_t)$ (where time dependence is suppressed for the discounted cost criterion; see the theorems for precise statements and notation) through an indirect approach: By discrete-time approximation serving as an intermediate step, and using the theory of $N$-exchangeable processes and their limits to exchangeable processes \cite{sanjarisaldiy2024optimality}, 
we are able to obtain such a near optimality result (of the McKean-Vlasov solution for the $N$-particle problem). Compared with \cite{lacker2017limit} the conditions on near optimality of McKean-Vlasov policies (and those which are Markov in discrete-time) for large $N$ are complementary; notably a convexity condition is not needed. 

A uniform discrete-time approximation result, over the number $N$ of particles, presented in Theorem \ref{TNconvevaluefunc} serves as a critical supporting result in our analysis. 

%Using our results above, in Theorems \ref{nearOptDiscFiniteHor} and %\ref{nearOptDiscDiscount}, we obtain an approximation theorem for optimal solutions %to large interacting particle systems by showing that the solution of the discrete-%time controlled McKean-Vlasov model problem is near optimal for such interacting %models, where we also build on related results in discrete-time  %\cite{sanjarisaldiy2024optimality}. 

\item We finally study the applicability of finite model approximations via establishing the weak Feller regularity of the discrete-time approximation where the control policy space is endowed with a Young topology. Accordingly, Theorem \ref{thm:finite_state_near_discrete} shows that a finite model which approximates the discrete-time McKean-Vlasov problem (which in turn approximates the continuous-time problem) is near optimal for the $N$-agent interacting particle diffusion model.

\item Our paper thus offers a unified approximation theory both in time and the number of agents and in the measure-valued state. This is summarized in Figure \ref{fig:approximation_roadmap}. 
\end{itemize}

\begin{figure}[H]
\centering
\begin{tikzpicture}[
node distance=2.9cm and 3.8cm,
every node/.style={font=\small},
box/.style={
draw,
rounded corners,
align=center,
minimum width=4.6cm,
minimum height=1.8cm
},
arrow/.style={
->,
thick,
>=latex
},
diag/.style={
->,
thick,
dashed,
>=latex
}
]

% Nodes
\node[box] (MVc) {
\textbf{Continuous-time}\\
McKean--Vlasov Control\\[1mm]
Optimal value: $V_T$
};

\node[box, below=of MVc] (MVd) {
\textbf{Discrete-time}\\
McKean--Vlasov Control\\[1mm]
Optimal value: $V_T^{h}$
};

\node[box, right=of MVc] (Nc) {
\textbf{Continuous-time}\\
$N$-Particle System\\[1mm]
Optimal value: $V_T^{N}$
};

\node[box, right=of MVd] (Nd) {
\textbf{Discrete-time}\\
$N$-Particle System\\[1mm]
Optimal value: $V_T^{N,h}$
};

\node[box, below=of MVd] (FS) {
\textbf{Finite-state}\\
Approximation\\[1mm]
Approximating policy: $\hat{U}^{h,n}$
};

% Left vertical arrow
\draw[arrow] (MVd) -- node[left,font=\scriptsize,xshift=-1mm] {
\begin{tabular}{c}
$\displaystyle \abs{\cJ_{T}^{\tilde{U}^{*,h}}(x) - V_T(x)}  \to 0$\\
as $h\downarrow0$\\[1mm]
Theorem~\ref{TNearFinite}: \\
Discrete-time MV optimal policy\\
is near-optimal for\\
continuous-time MV problem 
\end{tabular}
} (MVc);

% Bottom horizontal arrow
\draw[arrow] (MVd) -- node[below,font=\scriptsize] {
\begin{tabular}{c}
$\displaystyle \abs{V_T^h(x)- V_T^{N,h}(x)} \to 0,$ as $N\to \infty$\\ [1mm]
\cite[Theorems 2 and 4]{sanjarisaldiy2024optimality}:\\
In discrete-time MV optimal value
converges to the optimal value of the $N$-particle system
\end{tabular}
} (Nd);

% Top horizontal arrow
\draw[arrow] (MVc) -- node[above,font=\scriptsize] {
\begin{tabular}{c}
$\displaystyle |V_T(x)-V_T^N(x)| \to 0$ \, as $N\to\infty$\\[1mm]
Theorem~\ref{thm:Nparticlelimits}:\\
Continuous-time MV optimal policy\\
is near-optimal for\\
continuous-time $N$-particle system
\end{tabular}
} (Nc);

% Right vertical arrow
\draw[arrow] (Nd) -- node[right,font=\scriptsize,xshift=1mm] {
\begin{tabular}{c}
$\displaystyle |V^{N,h}(x)-V^N(x)| \to 0$\\
as $h\downarrow0$\\[1mm]
Theorem~\ref{TNconvevaluefunc}:\\
Discrete-time $N$-particle\\
system  optimal value converges\\
to the continuous-time $N$-particle \\
system optimal value.
\end{tabular}
} (Nc);

% Diagonal arrow
\draw[diag] (MVd) -- node[sloped,above,font=\scriptsize] {
\begin{tabular}{c}
$\displaystyle \bigl|\cJ^N_{T}(x,{\bf U}^{N,h}) - V_T^N(x)\bigr| \to 0$\\ [1mm]
Theorem~\ref{nearOptDiscFiniteHor}:\\
Discrete-time MV optimal policy\\
is near-optimal for\\
continuous-time $N$-particle system.
\end{tabular}
} (Nc);

% Finite-state approximation arrow
\draw[arrow] (FS) -- node[right,font=\scriptsize,xshift=1mm] {
\begin{tabular}{c}
$\displaystyle \bigl|
\cJ_{T,h}^{\hat U^{h,n}}(x)-V_T^h(x)
\bigr|
\to 0\,.$\\ [1mm]
Theorem~\ref{thm:finite_state_near_discrete}:\\
Finite-state optimal policy\\
is near optimal for the discrete-time\\
McKean--Vlasov control problem
\end{tabular}
} (MVd);

\end{tikzpicture}

\caption{Approximation results of the paper}

\label{fig:approximation_roadmap}
\end{figure}
%%%%%%%%%%%%%%%%%%%%%
%The remainder of this paper is organized as follows. Section %\ref{Scontinuous} introduces the continuous-time McKean-%Vlasov control problem. Section \ref{Sdiscrete} presents the %discrete-time approximation. Section \ref{Sconvergence} %analyzes the convergence of the discrete-time approximation. %Section \ref{Sconclusion} concludes and discusses future %research directions. \footnote{SP: Will update this later}
%%%%%%%%%%%%%%%%%%%%%%%%%%%%%%%%%%%%%%%%%%%%%%%%%%%%%%%%%%%%%%%%%%%%%%%%%%%%%%%%%%%%%%
\section{Optimal Control of McKean-Vlasov Diffusions} \label{Scontinuous}
McKean-Vlasov diffusions, also known as mean-field diffusions, are stochastic processes whose dynamics are influenced by their own probability distribution. In the context of controlled systems, we consider McKean-Vlasov \textit{controlled} diffusions, where the drift depends not only on the state $X_t$ and control $U_t$ at time $t$, but also on the law (probability distribution) of the state $X_t$, denoted by $\sL(X_t)$ and the diffusion matrix $\upsigma$ depends on state $X_t$ and its law $\sL(X_t)$\,.

Let $(\Omega, \sF, (\sF_t)_{t \geq 0}, \sP)$ be a filtered probability space satisfying the usual conditions and $W$ be a $d$-dimensional standard Brownian motion defined on this space. Let $\pV = \sP(\Act)$ be the space of all probability measures on $\Act$ with topology weak convergence. We assume that the action space $(\Act, d_{\Act})$ is a compact metric space. Consider a controlled diffusion process $X$ taking values in $\Rd$, with dynamics given by the stochastic differential equation (SDE):

\begin{equation} \label{E1.1}
dX_t = b(X_t, \sL(X_t), U_t) dt + \upsigma(X_t, \sL(X_t)) dW_t, \quad X_0 = x \in \mathbb{R}^d,
\end{equation}
where
\begin{itemize}
\item $b: \Rd \times \sP(\Rd) \times \Act \to \Rd$ is the drift term. We extend the drift term $b : \Rd \times \sP(\Rd) \times \pV \to \Rd$ as follows:
\[
b(x, \mu, v) = \int_{\Act} b(x, \mu, \zeta)\, v(d\zeta), \quad \text{for } v \in \pV.
\]
\item $\upsigma: \Rd \times \sP(\Rd) \to \RR^{d \times d}$ is the diffusion matrix.
\item $U$ is a $\pV$ valued process satisfying the following non-anticipativity condition: for $s<t\,,$ $W_t - W_s$ is independent of
$$\sF_s := \,\,\mbox{the completion of}\,\,\, \sigma(X_0, \mathcal{L}(X_0), U_r, W_r : r\leq s)\,\,\,\mbox{relative to} \,\, (\sF, \sP)\,.$$  
The process $U$ is called an \emph{admissible} control, and the set of all admissible controls is denoted by $\Uadm$
\item $\sL(X_t)$ denotes the law (probability distribution) of $X_t$.
\item $\sP(\Rd)$ denotes the space of probability measures on $\Rd$.
\end{itemize}
To ensure existence and uniqueness of solutions of (\ref{E1.1}), we impose the following assumptions on the drift $b$ and the diffusion matrix $\upsigma$\,.

\begin{assumption} \label{Alipschitz}
The function $b$ is jointly continuous and both $b$ and $\upsigma$ are Lipschitz continuous in their first two arguments. That is, there exists a constant $C_1 > 0$ such that for all $x, y \in \mathbb{R}^d$, $\mu, \nu \in \mathcal{P}(\mathbb{R}^d)$ and $\zeta \in \Act$\,:

\begin{align*}
  |b(x, \mu, \zeta) - b(y, \nu, \zeta)| +  |\upsigma(x, \mu) - \upsigma(y, \nu)| &\leq C_1(|x - y| + W_1(\mu, \nu)),
\end{align*}
where $W_1(\mu, \nu)$ is the Wasserstein-1 distance between probability measures $\mu$ and $\nu$, defined as:
\[
W_1(\mu, \nu) = \inf_{\pi \in \Pi(\mu, \nu)} \int_{\mathbb{R}^d \times \mathbb{R}^d} |x - y| d\pi(x, y),
\]
where $\Pi(\mu, \nu)$ is the set of all probability measures on $\mathbb{R}^d \times \mathbb{R}^d$ with marginals $\mu$ and $\nu$.
\end{assumption}

\begin{assumption} \label{Abounded}
The functions $b$ and $\upsigma$ are uniformly bounded.  That is, there exists a constant $C_2 > 0$ such that for all $x \in \mathbb{R}^d$, $\mu \in \mathcal{P}(\mathbb{R}^d)$ and $\zeta \in \Act$:

\begin{align*}
|b(x, \mu, \zeta)| + |\upsigma(x, \mu)| \leq C_2,
\end{align*}
\end{assumption}

\begin{assumption} \label{Acost}
The running cost function $c: \mathbb{R}^d \times \mathcal{P}(\mathbb{R}^d) \times \Act \to \mathbb{R}$ is jointly continuous and both $c$ and the terminal cost function $c_T: \mathbb{R}^d \times \mathcal{P}(\mathbb{R}^d) \to \mathbb{R}$ are uniformly bounded and Lipschitz continuous, with $c$ being Lipschitz continuous in first two arguments. That is, there exist constants $C_3, C_4 > 0$ such that for all $x, y \in \mathbb{R}^d$, $\mu, \nu \in \mathcal{P}(\mathbb{R}^d)$ and $\zeta \in \Act$:
\begin{align*}
&|c(x, \mu, \zeta) - c(y, \nu, \zeta)| + |c_T(x, \mu) - c_T(y, \nu)| \leq C_3(|x - y| + W_1(\mu, \nu)) \\
& |c(x,\mu,\zeta)| + |c_T(x,\mu)| \le C_4\,.
\end{align*}
\end{assumption}

\begin{remark}\label{rem:cost-continuity}
The Lipschitz-in-$(x,\mu)$ regularity of the running and terminal costs assumed in this paper is \emph{not} essential for near-optimality results.  In fact, the arguments may be carried out under the substantially weaker hypothesis that the costs are continuous in their arguments and satisfy suitable growth bounds (guaranteeing uniform integrability of the involved state laws). 

However, the stronger Lipschitz continuity hypotheses simplify several parts of the analysis and provide \emph{stronger} error bounds for example, the $O(h^{1/2})$ convergence rate of the value functions over the space of the piece-wise constant policies (see, Corollary~\ref{cor:EMAppro1AMV}, Corollary~\ref{cor:CNearFinite}, Corollary~\ref{Cor:forNparticle1a} and Lemma~\ref{lem:cost_consistency}).
\end{remark}

Under Assumption~\ref{Alipschitz} and Assumption~\ref{Abounded}, arguing as in \cite[Theorem~4.21]{CarmonaDelarue18I}, we have that there exists a unique strong solution to (\ref{E1.1}) for each $U\in \Uadm$\,.

%\sy{Our approach appears to hold also if we only impose unique weak solutions. However, the existence from Girsanov reduced model is more tedious due to the fixed point analysis; this has been done in the literature. Let's discuss if this is worth the effort.  On the other hand, under the approach considered, we can also have $\sigma$ depend on control.}

Here, the objective of the controller is to minimize total cost by choosing appropriate control policy. We are interested in the following cost evaluation criteria: \\

\noindent\textbf{Finite Horizon Cost Criterion:} For $U \in\Uadm$, the associated \emph{Finite horizon cost} is given by
\begin{equation} \label{EFcost}
\cJ_T^{U}(x) = \mathbb{E}_{x}^{U} \left[ \int_0^T c(X_t, \mathcal{L}(X_t), U_t) dt + c_T(X_T, \mathcal{L}(X_T)) \right],
\end{equation}
where $T > 0$ is a fixed terminal time and $X_{\cdot}$ is the solution of the (\ref{E1.1}) corresponding to $U\in\Uadm$ and $\Exp_x^{U}$ is the expectation with respect to the law of the process $X_{\cdot}$ with initial condition $x$. The McKean-Vlasov control problem is to find an admissible control $U^*$ that minimizes $\cJ_T^U(x)$, i.e., it satisfies
\begin{equation}\label{EFoptimal_control}
V_T(x) := \inf_{U \in \Uadm} \cJ_T^{U}(x) = \cJ_T^{U^*}(x)\,.
\end{equation}
Here $V_T(x)$ is called the value function and $U^*$ is an optimal policy.\\ \\
\textbf{Discounted Cost Criterion:} For $U \in\Uadm$, the associated \emph{$\alpha$-discounted cost} is given by
\begin{equation}\label{EDiscountedcost}
\cJ_{\alpha}^{U}(x) \,:=\, \Exp_x^{U} \left[\int_0^{\infty} e^{-\alpha s} c(X_s,\mathcal{L}(X_s), U_s) \D s\right],\quad x\in\Rd\,,
\end{equation} where $\alpha > 0$ is the discount factor\,. Here the controller tries to minimize (\ref{EDiscountedcost}) over the set of admissible controls $U\in \Uadm$\,.
A control $\bar{U}^{*}\in \Uadm$ is said to be an optimal control if for all $x\in \Rd$ 
\begin{equation}\label{EOpDiscost}
\cJ_{\alpha}^{\bar{U}^*}(x) = \inf_{U\in \Uadm}\cJ_{\alpha}^{U}(x) \,\,\, (:= \,\, V_{\alpha}(x))\,.
\end{equation}
%%%%%%%%%%%%%%%%%%%%%%%%%%%%
\section{Continuity of Cost in Control Policy and Existence of Optimal Policies} \label{Sdiscrete}

\subsection{Topology on control policy}
In order to utilize the weak convergence technique, we introduce the relaxed control representation of the control policies (for more details see \cite[Section~9.5]{KD92})\,. 
Let $\fB(\Act\times [0, \infty))$ be the $\sigma$-algebra of Borel subsets of $\Act\times [0, \infty)$\,. Then for any Borel measure $\hat{m}$ on $\fB(\Act\times [0, \infty))$, satisfying $\hat{m}(\Act\times [0, t]) = t$ for all $t\geq 0$, it follows that there exists a measure $\hat{m}_{t}$ on $\fB(\Act)$ such that $\hat{m}(\D\zeta, \D t) = \hat{m}_t(\D\zeta) \D t$\,. Let 
\begin{align*}
\fM(\infty) &\df \{\hat{m}: \hat{m}\,\,\text{is a Borel measure on}\,\,\fB(\Act\times [0, \infty))\,\,\text{satisfying}\,\, \hat{m}(\Act\times [0,t]) = t\nonumber\\
&\text{for any}\,\, t\geq  0,\,\, \text{with}\,\, \hat{m}_t\in \cP(\Act)\}\,.
\end{align*} The space $\fM(\infty)$ can be metrized by using the Prokhorov metric over $\cP(\Act\times [0, n])$ for $n\in\NN$\,. A sequence $\hat{m}^{n}$ converges to $\hat{m}$ in $\fM(\infty)$, if the normalized restriction of $\hat{m}^{n}$ converges weakly to the normalized restriction $\hat{m}$ on $\cP(\Act\times [0, k])$ for each $k\in \NN$\,. In particular, a sequence $\hat{m}^{n}$ converge to $\hat{m}$ in $\fM(\infty)$ if for any $\phi\in \cC_c(\Act\times [0, \infty))$ we have (this is known as compact weak topology)
$$\int_{\Act\times [0, \infty)} \phi(\zeta, t)\hat{m}^{n}(\D \zeta, \D t)\to \int_{\Act\times [0, \infty)} \phi(\zeta, t)\hat{m}(\D \zeta, \D t)\,.$$

Since $\cP(\Act\times [0, n])$ is complete, separable and compact for each $n\in\NN$, the space $\fM(\infty)$ inherits those properties\,. 

Now, define
\begin{align*}
\cM(\infty) &= \bigg\{m:(\Omega, \sF, \sP)\to \fM(\infty) \mid\,\, \text{for any}\,\,0\leq s <t < \infty,\,\,  m_{[0,s]}\,\,\text{is independent of}\,\, W_t - W_s\,\bigg\}\,.
\end{align*} Since the range space $\fM(\infty)$ is compact, any sequence in $\cM(\infty)$ is tight\,. Moreover, arguing as in \cite[Theorem~2.3.1]{ABG-book}, we have the following chattering lemma.
\begin{lemma}\label{LChatt}
Let $m\in\cM(\infty)$ be a relaxed control on probability space $(\Omega, \sF, \sP)$. Then there exists a sequence $m^{n}$ of non-anticipative piece-wise constant precise controls in $\cM(\infty)$ (defined on $(\Omega, \sF, \sP)$) such that, for each $T>0$ and $f\in \cC([0, T]\times \Act)$, we have
\begin{equation}\label{ELChatt1}
\int_{0}^{T}\int_{\Act}f(t, \zeta)m_{t}^{n}(d \zeta)d t \to \int_{0}^{T}\int_{\Act}f(t, \zeta)m_{t}(d \zeta)d t\quad\,\,\mbox{a.s. on}\,\,\, \Omega\,.
\end{equation}
\end{lemma}

\subsection{Existence of an optimal policy}

In this subsection, we consider the finite horizon criterion. Let $\cM(T)$ be the restriction of $\cM(\infty)$ on $[0, T]$. By similar argument as in \cite[Subsection~6.6]{CarmonaDelarue18I}, we have the following existence result.
\begin{theorem} \label{ThExistenceA}
Suppose Assumptions \ref{Alipschitz}, \ref{Abounded} and \ref{Acost} hold. Then there exists $m^*\in\cM(T)$ such that for all $x\in \Rd$
\begin{equation}\label{EThExistenceA}
V_T(x) = \cJ_T^{m^*}(x)\,.
\end{equation}
\end{theorem}
\begin{proof}
Let $m^n \to m$ in $\cM(T)$,\, and $X_n, X$ be the unique solution of (\ref{E1.1}) under $m^n, m$ respectively\,. Then, by the Cauchy-Schwarz inequality, we have

\begin{align*}
\sup_{t\leq T}\abs{X_t - X^n_t}^2 =& 3\Big(\sup_{t\leq T}\abs{\int_{0}^{t} b(X_s, \sL(X_s), m_s) ds - \int_{0}^{t} b(X_s, \sL(X_s), m_s^n) ds}^2 \nonumber\\
  &+ \sup_{t\leq T} \abs{\int_{0}^{t} b(X_s, \sL(X_s), m_s^n) ds - \int_{0}^{t} b(X_s^n, \sL(X_s^n), m_s^n) ds}^2 \nonumber\\
  & +  \sup_{t\leq T}\abs{\int_{0}^{t} \upsigma(X_s, \sL(X_s)) d W_s - \int_{0}^{t}\upsigma(X_s^n, \sL(X_s^n)) d W_s}^2\Big)
\nonumber\\
\leq & 3\Big(\sup_{t\leq T}\abs{\int_{0}^{t} b(X_s, \sL(X_s), m_s) ds - \int_{0}^{t} b(X_s, \sL(X_s), m_s^n) ds}^2 \nonumber\\
 &+ T\int_{0}^{T} \abs{b(X_s, \sL(X_s), m_s^n) ds - b(X_s^n, \sL(X_s^n), m_s^n)}^2 ds \nonumber\\
 & +  \sup_{t\leq T}\abs{\int_{0}^{t} \left(\upsigma(X_s, \sL(X_s)) - \upsigma(X_s^n, \sL(X_s^n))\right) d W_s}^2\Big)
\end{align*}
Since $b$, $\upsigma$ are bounded, by the Burkholder-Davis-Gundy (BDG) inequality (\cite[p.33]{ABG-book}), we obtain
\begin{align*}
\mathbb{E}[\sup_{t\leq T}\abs{X_t - X_t^n}^2] &\leq 3\Big( \hat{\eta}_{b,n} + T\mathbb{E}[\int_{0}^{T} \abs{b(X_s, \sL(X_s), m_s^n) ds - b(X_s^n, \sL(X_s^n), m_s^n)}^2 ds] \nonumber\\
  & +  \hat{C}_2\mathbb{E}[\int_{0}^{T} \abs{\left(\upsigma(X_s, \sL(X_s)) - \upsigma(X_s^n, \sL(X_s^n))\right)}^2 d s \Big)\nonumber\\
&\leq 3\Big(\hat{\eta}_{b,n} + 2C_1^2(T + \hat{C}_2)\mathbb{E}[\int_{0}^{T} \left(\abs{X_s -X_s^n}^2 + W_1(\sL(X_s), \sL(X_s^n))^2 \right) ds]\Big),  
\end{align*} where the constant $\hat{C}_2 >0$ is obtained from the BDG inequality and
\begin{align*}
\hat{\eta}_{b,n} = \mathbb{E}[\sup_{t\leq T}\abs{\int_{0}^{t} b(X_s, \sL(X_s), m_s) ds - \int_{0}^{t} b(X_s, \sL(X_s), m_s^n) ds}^2]\,.
\end{align*}

This implies 
\begin{align}\label{EThExistenceB}
\mathbb{E}[\sup_{t\leq T}\abs{X_t - X_t^n}^2] &\leq 3\Big(\hat{\eta}_{b,n} + 4C_1^2(T + \hat{C}_2)\int_{0}^{T}\mathbb{E}[\sup_{s\leq t}\abs{X_s - X_s^n}^2] dt\Big),
\end{align}

Thus, by the Gronwall's inequality, from (\ref{EThExistenceB}) we deduce that 
\begin{equation}\label{EThExistenceC}
\mathbb{E}[\sup_{t\leq T}\abs{X_t - X_t^n}^2] \leq 3\hat{\eta}_{b,n}e^{12TC_1^2(T + \hat{C}_2)}
\end{equation}
%%%%%%%%%%%%%%%%%%%%
We next analyze the convergence of $\hat{\eta}_{b,n}$ under weak convergence of relaxed controls in the Young topology. Let $$g_n(t)
:=
\int_0^t b(X_s,\sL(X_s),m_s^n)\,ds,
\qquad
g(t)
:=
\int_0^t b(X_s,\sL(X_s),m_s)\,ds.$$
Since $m^n \to m$ in $\cM(T)$, using Skorokhod representation theorem, we have $m^n \to m$ weakly (almost surely) along a subsequence. Now, since $b$ is bounded and continuous, for every fixed $t\in[0,T]$, $g_n(t) \to g(t), \quad \text{a.s.}$

Moreover, since $b$ is bounded, for all $s,t\in[0,T]$ with $s \leq t$, we have
\[
|g_n(t)-g_n(s)|
\le
\|b\|_\infty |t-s|,
\qquad
|g(t)-g(s)|
\le
\|b\|_\infty |t-s|,
\]
Thus, the family $\{g_n\}_{n\ge1}\cup\{g\}$ is equicontinuous on $[0,T]$. Since $g_n(t)\to g(t)$ pointwise for every $t\in[0,T]$, and the family is equicontinuous, a standard compactness argument implies
\[
\sup_{t\le T}|g_n(t)-g(t)|
\longrightarrow 0,
\qquad \text{a.s.}
\]
Therefore,
\[
\sup_{t\le T}
\left|
\int_0^t b(X_s,\sL(X_s),m_s^n)\,ds
-
\int_0^t b(X_s,\sL(X_s),m_s)\,ds
\right|^2
\longrightarrow 0,
\qquad \text{a.s.}
\]
Finally, boundedness of $b$ yields
\[
\sup_{t\le T}|g_n(t)-g(t)|^2
\le
4T^2\|b\|_\infty^2,
\]
hence the dominated convergence theorem gives
\begin{equation}\label{EThExistenceC1A}
\hat{\eta}_{b,n} = \mathbb E\!\left[
\sup_{t\le T}
\left|
\int_0^t b(X_s,\sL(X_s),m_s^n)\,ds
-
\int_0^t b(X_s,\sL(X_s),m_s)\,ds
\right|^2
\right]
\longrightarrow 0.
\end{equation}
Therefore, letting $n\to \infty$, from (\ref{EThExistenceC}), we obtain 
\begin{equation}\label{EThExistenceD}
\lim_{n\to\infty}\mathbb{E}[\sup_{t\leq T}\abs{X_t - X_t^n}^2] = 0\,.
\end{equation}
Moreover, since $c, c_T$ are Lipschitz continuous (see Assumption~\ref{Acost}), we have
\begin{align}\label{EcostConver1A}
\abs{\cJ_T^{m^n}(x) - \cJ_T^{m}(x)} =& \abs{\mathbb{E}_{x} \left[ \int_0^T c(X_t^n, \sL(X_t^n), m^n_t) dt + c_T(X_T^n, \sL(X_T^n)) \right] \nonumber\\
&- \mathbb{E}_{x} \left[ \int_0^T c(X_t, \sL(X_t), m_t) dt + c_T(X_T, \sL(X_T)) \right]}\nonumber\\
\leq & \mathbb{E}_{x} \left[ \int_0^T \abs{\left(c(X_t^n, \sL(X_t^n), m^n_t) - c(X_t, \sL(X_t), m_t^n)\right)} dt \right] \nonumber\\
&+ \mathbb{E}_{x} \left[\abs{c_T(X_T^n, \sL(X_T^n)) - c_T(X_T, \sL(X_T))} \right]\nonumber\\
&+\abs{\mathbb{E}_{x}\left[\int_0^T \Big(c(X_t, \sL(X_t), m^n_t) - c(X_t, \sL(X_t), m_t)\Big) dt\right]}\nonumber\\
\leq & \sqrt{T}\sqrt{\mathbb{E}_{x} \left[ \int_0^T \abs{\left(c(X_t^n, \sL(X_t^n), m^n_t) - c(X_t, \sL(X_t), m_t^n)\right)}^2 dt \right]} \nonumber\\
&+ \sqrt{\mathbb{E}_{x} \left[\abs{c_T(X_T^n, \sL(X_T^n) - c_T(X_T, \sL(X_T))}^2\right]} +\hat{\eta}_{c,n}\nonumber\\
\leq & \sqrt{T}\sqrt{\mathbb{E}_{x} \left[ C_3^2\int_0^T \abs{\left(\abs{X_t^n -X_t}+W_1\left(\sL(X_t^n), \sL(X_t)\right)\right)}^2 dt \right]} \nonumber\\
&+ \sqrt{\mathbb{E}_{x} \left[C_3^2\abs{\left(\abs{X_T^n -X_T}+W_1\left(\sL(X_T^n), \sL(X_T)\right)\right)}^2\right]} +\hat{\eta}_{c,n}\nonumber\\
\leq & \sqrt{T}C_3\sqrt{2\mathbb{E}_{x} \left[T\sup_{t\leq T} \abs{X_t^n -X_t}^2\right]} + C_3\sqrt{2\mathbb{E}_{x} \left[\abs{X_T^n -X_T}^2\right]} +\hat{\eta}_{c,n}\nonumber\\
\leq & \sqrt{2}C_3(T + 1)\sqrt{\mathbb{E}[\sup_{t\leq T}\abs{X_t - X_t^n}^2]} + \hat{\eta}_{c,n}\,,
\end{align} 
where $\hat{\eta}_{c,n} := \abs{\mathbb{E}_{x}\left[\int_0^T \Big(c(X_t, \sL(X_t), m^n_t) - c(X_t, \sL(X_t), m_t)\Big) dt\right]}$\,.

Again, by the Skorokhod representation theorem and the dominated convergence theorem (as in (\ref{EThExistenceC1A})), we see that the second term of the above inequality converges to zero. Thus, from (\ref{EThExistenceD}), we obtain the map $m\mapsto \cJ(m, x)$ is continuous on $\cM(T)$. Since $\cM(T)$ is compact, we have there exists $m^*\in \cM(T)$  such that (\ref{EThExistenceA}) holds. This completes the proof of the theorem\,.
\end{proof}
\begin{remark}
In the proof of the above theorem, we have shown that $m\mapsto \cJ^{m}(x)$ continuous on $\cM(T)$. In addition to that we have shown that for any $T>0$, the map $m\in \cM(T)$ goes to the associated solution of the equation (\ref{E1.1}) is continuous in quadratic mean (in particular in probability)\,.

\end{remark}
\begin{remark}
The PDE / HJB / master-equation approach \cite{Cosso2021MasterBE}, lifts the problem to the space of probability measures and derives a dynamic-programming PDE (or the master equation) for the value; this approach provides existence and uniqueness of value function, verification theorems and existence of optimal feedbacks policies, but it requires much stronger structural and regularity assumptions, e.g., smoothness of the drift (uniform Lipschitz continuity) and diffusion matrix (continuously differentiable) (see \cite{Cosso2021MasterBE}), uniform nondegeneracy or strong ellipticity, in certain settings monotonicity or convexity. These assumptions ensure that the lifted HJB/master equation admits classical (or well-behaved viscosity) solutions \cite{Cosso2021MasterBE}.   
\end{remark}
%%%%%%%%%%%%%%%%%%%%%%%
\subsection{Near optimality of piece-wise constant policies}
In view of the above continuity and denseness results we have following near-optimality result.
\begin{lemma}\label{LNearOptiA}
Suppose Assumptions \ref{Alipschitz}, \ref{Abounded} and \ref{Acost} hold. Then for any $\epsilon > 0$ there exists $m^{*\epsilon}\in\cM(T)$ which is piece-wise constant such that for all $x\in \Rd$
\begin{equation}\label{ELNearOptiA}
\cJ_T^{m^{*\epsilon}}(x) \leq V_T(x) + \epsilon \,.
\end{equation}
\end{lemma} 
\begin{proof}
From the proof the Theorem~\ref{ThExistenceA}, we have the map $m\mapsto \cJ_T^{m}(x)$ continuous on $\cM(T)$. Moreover, Lemma~\ref{LChatt}, implies that the space of non-anticipative piece-wise constant policies are dense in $\cM(T)$. Since $V_T(x) = \cJ_T^{m^*}(x)$, above continuity and denseness results ensure existence of a piece-wise constant policy satisfying (\ref{ELNearOptiA})\,. This completes the proof\,.
\end{proof}

Similarly, for the discounted cost case, we have the following near-optimality result.

\begin{lemma}\label{LNearOptiDiscounted1A}
Suppose Assumptions \ref{Alipschitz}, \ref{Abounded} and \ref{Acost} hold. Then for any $\epsilon > 0$ there exists $m^{*\epsilon}\in\cM(\infty)$ which is piece-wise constant such that for all $x\in \Rd$
\begin{equation*}
\cJ_{\alpha}^{m^{*\epsilon}}(x) \leq V_{\alpha}(x) + \epsilon \,.
\end{equation*}
\end{lemma} 
\begin{proof}
Let $m^n\to m$ in $\cM(\infty)$ and and $X_n, X$ be the unique solution of (\ref{E1.1}) under $m^n, m$ respectively\,. Since $c, c_T$ are bounded and Lipschitz continuous, we have
\begin{align*}
\abs{\cJ_{\alpha}^{m^n}(x) - \cJ_{\alpha}^{m}(x)} =& \abs{\mathbb{E}_{x} \left[ \int_0^T e^{-\alpha t}\left(c(X_t^n, \sL(X_t^n), m^n_t) -c(X_t, \sL(X_t), m_t)\right) dt\right] \nonumber\\
&+ \mathbb{E}_{x} \left[ \int_T^{\infty}e^{-\alpha t}\left(c(X_t^n, \sL(X_t^n), m^n_t) -c(X_t, \sL(X_t), m_t)\right) dt\right]}\nonumber\\
\leq & \abs{\mathbb{E}_{x} \left[ \int_0^T e^{-\alpha t}\left(c(X_t^n, \sL(X_t^n), m^n_t) -c(X_t, \sL(X_t), m_t)\right) dt\right]} \nonumber\\
&+ 2\frac{\|c\|_{\infty}}{\alpha}e^{-\alpha T}
\end{align*}
Hence, for any $\epsilon > 0$, we can find $T_{\epsilon} > 0$ such that
\begin{align}\label{EConvergenceConti1A}
\abs{\cJ_{\alpha}^{m^n}(x) - \cJ_{\alpha}^{m}(x)}
\leq & \abs{\mathbb{E}_{x} \left[ \int_0^{T_{\epsilon}} e^{-\alpha t}\left(c(X_t^n, \sL(X_t^n), m^n_t) -c(X_t, \sL(X_t), m_t)\right) dt\right]} + \frac{\epsilon}{2}
\end{align} Thus, following the arguments as in the proof of Theorem~\ref{ThExistenceA}, we have the map $m\mapsto \cJ_{\alpha}^{m}(x)$ is continuous on $\cM(\infty)$\,. By the density of the piece-wise constant policies, we deduce that for any $\epsilon > 0$, there exists a piece-wise constant policy $m^{*\epsilon}\in \cM(\infty)$ such that
\begin{equation*}
\cJ_{\alpha}^{m^{*\epsilon}}(x) \leq V_{\alpha}(x) + \epsilon \,.
\end{equation*}    
\end{proof}

%%%%%%%%%%%%%%%%%%%%%%%%%%%%%%%%%%%%%%
\section{Discrete-Time Approximate McKean-Vlasov Model and Near Optimality}\label{sec:discreteapproxiMcVla}

In this section, using piece-wise constant policies we will construct a sequence of discrete-time controlled processes which will then approximate the solutions of (\ref{E1.1}). 

\subsection{Discrete-Time Approximate Model} \label{SdiscreteSDE}

Let $h >0$ and define $t_k = k h$ for $k = 0, 1, \dots ,$. Now, we approximate the SDE \eqref{E1.1} using the following Euler-Maruyama scheme. Given a piece-wise constant $U^h$ (as in \eqref{Econtrol}), the Euler-Maruyama type approximation for the controlled McKean-Vlasov SDE (\ref{E1.1}) is defined as follows:

\begin{equation} \label{Eeuler_maruyama}
\begin{aligned}
X_{t_{i+1}}^h &= X_{t_i}^h + b(X_{t_i}^n, \sL(X_{t_i}^h), U_{t_i}^h) h + \upsigma(X_{t_i}^h, \sL(X_{t_i}^h)) \Delta W_i, \\
X_0^h &= x,
\end{aligned}
\end{equation}
where
\begin{itemize}
\item $X_0^h = x$.
\item $U_{t_i}^h$ is the discrete-time control at time $t_i$, taking values in $\pV$.
\item $\Delta W_i = W_{t_{i+1}} - W_{t_i}$ are independent and identically distributed (i.i.d.) Gaussian random vectors with mean 0 and covariance matrix $(h)I_d$, where $I_d$ is the $d \times d$ identity matrix.
\item $\sL(X_{t_i}^h)$ denotes the law of $X_{t_i}^h$. 
\end{itemize}

For simplicity, in the following we denote the samples $X_{t_i}^h$ by $X_{i}^h$\,. We note that in practice, the law of $X_{t_i}^h$ would be estimated from samples of $X_{t_i}^h$.

\begin{remark}\label{ER4.1}
To make this scheme implementable, we need to approximate the law $\sL(X_i^h)$. A common approach is to use a particle system. Let $X_i^{h,j}$, $j=1, \ldots, N$, be $N$ independent samples satisfying
\begin{equation}
\label{Eparticle_system}
\begin{cases}
X_0^{h,j} = x, \\
X_{i+1}^{h,j} = X_i^{h,j} + b(X_i^{h,j}, \hat{\mu}_i^h, U_{t_i}^h) h + \sigma( X_i^{h,j}, \hat{\mu}_i^h) \Delta W_i^j,
\end{cases}
\end{equation}
where $\Delta W_i^j = W_{t_{i+1}}^j - W_{t_i}^j$ and $W^j$ are independent copies of the Brownian motion $W$, and $\hat{\mu}_i^h = \frac{1}{N} \sum_{j=1}^N \delta_{X_i^{h,j}}$ is the empirical measure. 
\end{remark}

Under this discrete-time setup, the associated discrete-time cost evaluation criteria are given by:\\ \\
\textbf{Discrete-time Finite Horizon Cost:}
\begin{equation} \label{Ediscrete_cost}
\cJ_{T,h}^{U^h}(x) = \mathbb{E}_{x}^{U^h} \left[ \sum_{k=0}^{N_T^h-1} c(X_k^h, \mathcal{L}(X_k^h), U_k^h) h + c_T(X_{N_T^h}^h, \sL(X_{N_T^h}^h)) \right],
\end{equation}
where $N_T^h := \sup\{n: nh\leq T \}$ and $U^h = (U_0^h, U_1^h, \dots )$ is a sequence of controls. For the discrete-time model, an {\em admissible policy} is a sequence of control functions $\{U^h_k,\, k\in \ZZ_{+}\}$ such that $U^h_k$ is measurable with respect to the $\sigma$-algebra generated by the information variables
$
I_k^h=\{X_{[0,k]}^h,\sL(X_{[0,k]}^h), U_{[0,k-1]}^h\},\,\, k \in \NN,\,\, I_0^h=\{X_0^h\},
$ that is
\begin{equation}
\label{Econtrol}
U_k^h=v_k^h(I_k^h),\quad k\in \ZZ_{+},
\end{equation}
where $v_k^h$ is a $\pV$-valued measurable function for $k\in \ZZ_{+}$\,. We define $\Uadm^h$ to be the set of all such admissible policies. Let $\Um^h := \{U^h\in \Uadm^h: U_k^h = v_k^h(X_k^h, \sL(X_k^h))\,\,\text{for some measurable map}\,\, v_k^h:\Rd\times \sP(\Rd)\to \pV\}$ be the set of all Markov policies\,. If the function $v_k^h$ does have an explicit time dependence, it is called stationary Markov policy, we denote it by $\Usm^h := \{U^h\in \Uadm^h: U_k^h = v^h(X_k^h, \sL(X_k^h))\,\,\text{for some measurable map}\,\, v^h:\Rd\times \sP(\Rd)\to \pV\}$\,. A policy $v^h\in \Usm^h$ is said to be a deterministic stationary Markov policy if $v^h(x, \mu) = \delta_{f(x, \mu)}$ for some measurable map $f:\Rd \times \sP(\Rd)\to \Act$. Let $\Udsm^h$ be the space of all deterministic stationary Markov policies\,.

The discrete-time McKean-Vlasov control problem is to find a control sequence $U^{h, *}$ that minimizes $\cJ_{T,h}^{U^h}(x)$: that is 
\begin{equation} \label{Ediscrete_optimal_control}
V_T^h(x) \df \inf_{U^h \in \Uadm^h} \cJ_{T,h}^{U^h}(x) = \cJ_{T,h}^{U^{h, *}}(x).
\end{equation}
Here $V_T^h(x)$ is the discrete-time value function. 

\begin{remark}
As a consequence of the construction, every element of $\Uadm^h$ lies in $\cM(\infty)$\,.
\end{remark}

We now show that the infimum is attained, in particular, arguing as \cite{pham2016discrete} we have there exists an optimal control $U^{h, *} \in \Um^h$ which is Markov such that $V_T^h(x) = \cJ_{T,h}^{U^{h, *}}(x)$. To this end, we also obtain a weak Feller continuity property, which will be utilized further later in the paper in Section \ref{FiniteModApp} in the context of finite model approximations. In particular, we have the following theorem

\begin{theorem}\label{TOptiDiscMcKV1A}
Suppose Assumptions \ref{Alipschitz}, \ref{Abounded} and \ref{Acost} hold. Then
\begin{itemize}
\item[(i)] The map \[ \Rd \times {\cal P}(\Rd) \times \mathbb{U} \ni (x,\mu,\zeta) \mapsto \bE[g(X^h_{k+1}) | X^h_{k}=x, {\cal L}(X^h_k)=\mu,U^h_{k}=\zeta ] \in \mathbb{R},\]
is continuous for every bounded continuous $g: \Rd \to \RR$. That is, the discrete-time kernel is weakly continuous (or weak Feller).
\item[(ii)] There exists $U^{h, *} \in \Um^h$ (which is Markov) such that
\begin{equation}\label{ETOptiDiscMcKV1A}
V_T^h(x) = \cJ_{T,h}^{U^{h, *}}(x)\,.
\end{equation}
\end{itemize}
\end{theorem} 
\begin{proof} Let $(x_n,\mu_n,\zeta_n)\to(x,\mu,\zeta)$ in $\Rd\times \mathcal{P}(\Rd)\times\Act$, where the convergence of $\mu_n$ to $\mu$ is in $W_1$. We must show, as $n\to \infty$
\begin{equation}\label{thm:Econv1A}
\mathbb{E}\left[g(X^h_{k+1})\mid X^h_{k} = x_n, \sL(X^h_{k}) = \mu_n, U^h_{k} = \zeta_n \right] \;\to\; \mathbb{E}\left[ g(X^h_{k+1})\mid X^h_{k} = x, \sL(X^h_{k}) = \mu, U^h_{k} = \zeta\right]
\end{equation}
Now let
\begin{align*}
\bar{X}^{h,n}_{k+1} &:= x_n + b(x_n,\mu_n,\zeta_n)h + \sigma(x_n,\mu_n)\,\Delta W_k,\\
X^h_{k+1}&:= x + b(x,\mu,\zeta)h + \sigma(x,\mu)\,\Delta W_k,
\end{align*}

Thus, we have
\[
\bar{X}^{h,n}_{k+1} - X^h_{k+1} = (x_n - x) + h\big(b(x_n,\mu_n,\zeta_n) - b(x,\mu,\zeta)\big)
     + \big(\sigma(x_n,\mu_n)-\sigma(x,\mu)\big)\,\Delta W_k.
\]
Using $(a+b+c)^2\le3(a^2+b^2+c^2)$, we get
\begin{align*}
\mathbb{E}\big[|\bar{X}^{h,n}_{k+1} - X^h_{k+1}|^2\big]
\le & 3|x_n - x|^2 + 3h^2 |b(x_n,\mu_n,\zeta_n)-b(x,\mu,\zeta)|^2\nonumber \\
 &\,\,\, + 3|\sigma(x_n,\mu_n)-\sigma(x,\mu)|^2\,\mathbb{E}\big[|\Delta W_k|^2\big].
\end{align*}
By the Lipschitz continuity of \(b,\sigma\) and since \(\mathbb{E}[|\Delta W_k|^2] = h\), we get
\begin{align}\label{thm:Econv1B}
\mathbb{E}\big[|\bar{X}^{h,n}_{k+1} - X^h_{k+1}|^2\big]
&\le 3|x_n - x|^2 + 2C_1(h^2 + h)\big(|x_n - x|^2 + W_1(\mu_n,\mu)^2\big)\nonumber\\
&\quad + 3h^2|b(x,\mu,\zeta_n)-b(x,\mu,\zeta)|^2.
\end{align}

Taking $n\to \infty$, it follows that the right hand side of (\ref{thm:Econv1B}) goes to $0$\,. Thus, $\bar{X}^{h,n}_{k+1} \to X^h_{k+1}$ weakly as $n\to \infty$\,. Therefore, by the Skorohod representation theorem and the dominated convergence theorem, we obtain (\ref{thm:Econv1A})\,.    

In view of the above weak Feller continuity result, dynamic programming recursions are well-defined via standard measurable selection conditions \cite{HernandezLermaMCP} as studied in \cite[Theorem 3]{sanjarisaldiy2024optimality}: One can define a controlled Markov model where $\mu_k = {\cal L}(X_k)$ is the state and the joint measure $\Theta_k = \{v \in {\cal P}(\mathbb{R}^d \times \mathbb{U}): v(dx\times \mathbb{U})=\mu_k(dx) \}$ serving as the {\it action} (endowed with the Young topology  \cite{balder1997consequences} \cite{bauerle2010optimal,bauerle2018optimal}\cite{florescu2012young}\cite{yuksel2023borkar}). In particular, $\Theta_k$ defines a control policy for $\mu_k$ almost every $x \in \Rd$ as a relaxed control policy. Then, one can show that $(\mu_k, \Theta_k)$ forms a controlled Markov chain and an optimal $\Theta_k$ is uniquely defined by $\mu_k$, leading to an optimal control. 
\end{proof}

The reader is also referred to \cite{pham2016discrete} for an earlier study which established the existence for such discrete-time models. 

\textbf{Discrete-time Discounted Cost:} For each $U^h\in \Uadm^h$, the associated discounted cost of the approximating discrete-time model is given by
\begin{equation}\label{EDiscreteDiscountA}
\cJ_{\alpha,h}^{U^h}(x, c) \,\df\, \Exp_x^{U^h} \left[\sum_{k=0}^{\infty} \beta^k hc(X_k^h, U_k^h) \mid X_{0}^h = x \right],\quad x\in\Rd\,,
\end{equation} where $\beta \df e^{-\alpha h}$\,. The optimal cost is defined as
\begin{equation}\label{EDiscreteDiscountB}
V_{\alpha}^h(x) = \cJ_{\alpha,h}^{*}(x, c) \,\df\, \inf_{U^h\in \Uadm^h}\cJ_{\alpha,h}^{U^h}(x, c)\,.
\end{equation}
%%%%%%%%%%%%%%%%%%%%%%%%%%%%%%%%%
For any discrete-time admissible policy $U^h\in \Uadm^h$ the interpolated continuous time policy $U^h(\cdot)$ is defined as
\begin{align}\label{EInterPol}
U^{h}(t)= U^h_k \text{ for } t\in[t_k, t_{k+1})\,.
\end{align}
The set of all interpolated continuous time admissible policies is denoted by $\bar{\Uadm}^h$ and set of all interpolated continuous time stationary Markov strategies is denoted by $\bar{\Uadm}_{\mathsf{sm}}^h$\,. Also, for any discrete-time Markov chain $\{X^h_k\}$, the associated continuous-time interpolated process $X^{h}(\cdot)$ is given by
\begin{align}\label{EInterState}
X^{h}(t)= X^h_k \text{ for } t\in[t_k, t_{k+1})\,.
\end{align} Thus, the sample paths of the continuous-time interpolated process $X^{h}(\cdot)$ are right continuous with left limits. This leads us to consider the function space 
\begin{equation*}
\sD(\Rd; 0, \infty)\df \{\omega:[0, \infty)\to \Rd \mid \omega(\cdot) \,\,\text{is right continuous and have left limits at every}\,\, t > 0\}\,.
\end{equation*} We will consider the space $\sD(\Rd; 0, \infty)$ endowed with the Skorokhod topology (for details see \cite[Section~16, p. 166]{PBill-book})\,.

%%%%%%%%%%%%%%%%%%%%%%%%%%%%%%%%%%%
\subsection{Near Optimality and Convergence Analysis} \label{Sconvergence}
This section is dedicated to proving the convergence of the discrete-time value functions $V^h(x)$ and $V_{\alpha}^h(x)$ to the continuous-time value function $V(x)$ and $V_{\alpha}(x)$ respectively as $h \to 0$. Furthermore, we aim to show that a discrete-time optimal control sequence $U^{h,*}$ are near optimal in the continuous-time optimal system as the parameter of the discretization approaches to zero.
%%%%%%%%%%%%%%%%%%%%%%%%%%%%%%%%%%%%%%%%%%%%%%%%%%%%
Now we prove the following convergence result which will play an important role in our discrete-time approximation problem.
\begin{theorem} \label{Thconvergence}
Suppose Assumptions \ref{Alipschitz} and \ref{Abounded} hold. Let $X_t$ be the solution of the controlled McKean-Vlasov SDE (\ref{E1.1}) under $U\in \Uadm$ and $X_{i}^{h}$ (with continuous time interpolated process $X^{h}(\cdot)$) be the solution of the Euler-Maruyama scheme (\ref{Eeuler_maruyama}) under the piece-wise constant policy $U^h\in \Uadm^h$ which approximates $U$ weakly a.s. (existence follows from Lemma~\ref{LChatt}). Then, we have
\begin{equation}\label{EConv1A}
 \lim_{h\downarrow 0}\mathbb{E}[\sup_{t \in [0, T]}|X_t - X^{h}(t)|^2] = 0,
\end{equation}
\end{theorem}

\begin{proof}
For any $t \leq T$, we have
\begin{align*}
X_t - X^h(t) =& \int_0^t b(X_s, \sL(X_s), U_s) ds + \int_0^t \upsigma(X_s, \sL(X_s)) dW_s - X^h(t),\nonumber\\
=& \int_{hN_t^h}^t b(X_s, \sL(X_s), U_s) ds + \int_{hN_t^h}^t \upsigma(X_s, \sL(X_s)) dW_s  \nonumber\\
  &+\int_{0}^{hN_t^h} b(X_s, \sL(X_s), U_s) ds - \sum_{k=0}^{N_t^h-1} b(X_{t_k}^h, \sL(X_{t_k}^h), U_{t_k}^h) h \nonumber\\
  & +  \int_{0}^{hN_t^h} \upsigma(X_s, \sL(X_s), U_s) d W_s - \sum_{k=0}^{N_t^h-1} \upsigma(X_{t_k}^h, \sL(X_{t_k}^h), U_{t_k}^h) \Delta W_k \nonumber\\
=& \int_{hN_t^h}^t b(X_s, \sL(X_s), U_s) ds + \int_{hN_t^h}^t \upsigma(X_s, \sL(X_s), U_s) dW_s  \nonumber\\
  &+\int_{0}^{hN_t^h} b(X_s, \sL(X_s), U_s) ds - \int_{0}^{hN_t^h} b(X_s, \sL(X_s), U^h(s)) ds \nonumber\\
  &+ \int_{0}^{hN_t^h} b(X_s, \sL(X_s), U^h(s)) ds - \int_{0}^{hN_t^h} b(X^h(s), \sL(X^h(s)), U^h(s)) ds \nonumber\\
  & +  \int_{0}^{hN_t^h} \upsigma(X_s, \sL(X_s)) d W_s - \int_{0}^{hN_t^h}\upsigma(X^h(s), \sL(X^h(s))) d W_s
\end{align*}
Thus, by the Cauchy-Schwarz inequality, we get
\begin{align*}
\sup_{t\leq T}\abs{X_t - X^h(t)}^2 =& 4\Big(\sup_{t\leq T}\abs{\int_{hN_t^h}^t b(X_s, \sL(X_s), U_s) ds + \int_{hN_t^h}^t \upsigma(X_s, \sL(X_s), U_s) dW_s}^2  \nonumber\\
  &+ \sup_{t\leq T}\abs{\int_{0}^{hN_t^h} b(X_s, \sL(X_s), U_s) ds - \int_{0}^{hN_t^h} b(X_s, \sL(X_s), U^h(s)) ds}^2 \nonumber\\
  &+ \sup_{t\leq T} \abs{\int_{0}^{hN_t^h} b(X_s, \sL(X_s), U^h(s)) ds - \int_{0}^{hN_t^h} b(X^h(s), \sL(X^h(s)), U^h(s)) ds}^2 \nonumber\\
  & +  \sup_{t\leq T}\abs{\int_{0}^{hN_t^h} \upsigma(X_s, \sL(X_s)) d W_s - \int_{0}^{hN_t^h}\upsigma(X^h(s), \sL(X^h(s))) d W_s}^2\Big)
\nonumber\\
\leq & 4\Big(\sup_{t\leq T}\abs{\int_{hN_t^h}^t b(X_s, \sL(X_s), U_s) ds + \int_{hN_t^h}^t \upsigma(X_s, \sL(X_s), U_s) dW_s}^2  \nonumber\\
  & + \sup_{t\leq T}\abs{\int_{0}^{hN_t^h} b(X_s, \sL(X_s), U_s) ds - \int_{0}^{hN_t^h} b(X_s, \sL(X_s), U^h(s)) ds}^2 \nonumber\\
  &+ T\int_{0}^{T} \abs{b(X_s, \sL(X_s), U^h(s)) ds - b(X^h(s), \sL(X^h(s)), U^h(s))}^2 ds \nonumber\\
  & +  \sup_{t\leq T}\abs{\int_{0}^{hN_t^h} \left(\upsigma(X_s, \sL(X_s)) - \upsigma(X^h(s), \sL(X^h(s)))\right) d W_s}^2\Big)
\end{align*}
Using the fact that $b$, $\upsigma$ are bounded and the Burkholder-Davis-Gundy (BDG) inequality
\begin{align}\label{EThconv1A}
&\mathbb{E}[\sup_{t\leq T}\abs{X_t - X^h(t)}^2] \nonumber \\
& \leq 4\Big(\hat{C}_1 h + \eta_{h} + T\mathbb{E}[\int_{0}^{T} \abs{b(X_s, \sL(X_s), U^h(s)) ds - b(X^h(s), \sL(X^h(s)), U^h(s))}^2 ds] \nonumber\\
  & +  \hat{C}_2\mathbb{E}[\int_{0}^{T} \abs{\left(\upsigma(X_s, \sL(X_s)) - \upsigma(X^h(s), \sL(X^h(s)))\right)}^2 d s \Big)\nonumber\\
&\leq 4\Big(\hat{C}_1 h + \eta_{h} + 2(T + \hat{C}_2)\mathbb{E}[\int_{0}^{T} \abs{X_s -X^h(s)}^2 + W_1(\sL(X_s), \sL(X^h(s)))^2 ds]\Big),  
\end{align} where $\hat{C}_1 > 0$ is a constant which depends only on the bounds of $b, \upsigma$,\, the constant $\hat{C}_2 >0$ is obtained from the BDG inequality and
\begin{align*}
\eta_{h} = \mathbb{E}[\sup_{t\leq T}\abs{\int_{0}^{hN_t^h} b(X_s, \sL(X_s), U_s) ds - \int_{0}^{hN_t^h} b(X_s, \sL(X_s), U^h(s)) ds}^2]\,.
\end{align*}
Since $W_1(\sL(X_s), \sL(X^h(s)))^2 \leq W_2(\sL(X_s), \sL(X^h(s)))^2 \leq \mathbb{E}[\abs{X_s - X^h(s)}^2]$, we obtain
\begin{align*}
\mathbb{E}[\sup_{t\leq T}\abs{X_t - X^h(t)}^2] &\leq 4\Big(\hat{C}_1 h + \eta_{h} + 4(T + \hat{C}_2)\int_{0}^{T}\mathbb{E}[\sup_{s\leq t}\abs{X_s - X^h(s)}^2]\Big).
\end{align*}

Now applying the Gronwall's inequality, it follows that
\begin{equation}\label{EThconv1B}
\mathbb{E}[\sup_{t\leq T}\abs{X_t - X^h(t)}^2] \leq 4(\hat{C}_1 h + \eta_{h})e^{16T(T + \hat{C}_2)}
\end{equation}
Since $U^h(\cdot) \to U$ weakly a.s. in $\Omega$, by dominated convergence theorem (as in (\ref{EThExistenceC1A})) we get $\eta_h \to 0$ as $h\downarrow 0$\,. Thus, by letting $h\downarrow 0$ from (\ref{EThconv1B}), we deduce that (\ref{EConv1A}) holds. This completes the proof of the theorem\,.
\end{proof}
%%%%%%%%%
In the above theorem, if we want to approximate the Mckean-Vlasov system under a fixed piece-wise constant non-anticipative policy, then $\eta_{h}$ becomes zero and we get the following rate of convergence.
\begin{corollary}\label{cor:EMAppro1AMV}
Suppose Assumptions \ref{Alipschitz} and \ref{Abounded} hold. Let $X_t$ be the solution of the controlled Mckean-Vlasov system (\ref{E1.1}) under a piece-wise constant non-anticipative policy $m^{h}$ and $X_{k}^{h}$ (with continuous time interpolated process $X^{h}(\cdot)$) be the solution of the Euler-Maruyama scheme (\ref{Eeuler_maruyama}) under the same piece-wise constant policy. Then, for some positive constant $\bar{C}_{EM}$ (independent of $m^{h}$) we have
\begin{equation}\label{ENConvCor1AMV}
\mathbb{E}[\sup_{t \in [0, T]}|X_t - X^{h}(t)|^2] \leq \bar{C}_{EM}h\,.
\end{equation}    
\end{corollary}
%%%%%%%%
\begin{remark}\label{rem:MV_EM_vs_sampled} Let \(T>0\) be fixed. For any given piece-wise constant policy \(\tilde{U}^h(t)= \sum_{k}U^h_k I_{[kh,(k+1)h)} (t)\), define the discrete-time process \emph{sampled} at grid points by
\begin{align}\label{DTApprSimul}
X_{(k+1)h} \;=\; X_{kh} + \int_{kh}^{(k+1)h} b\bigl(X_s,\sL(X_s),U^h_k\bigr)ds
+ \int_{kh}^{(k+1)h}\upsigma\bigl(X_s,\sL(X_s)\bigr)\,dW_s,
\end{align}
with continuous-time interpolated process $\tilde{X}^h$. The Euler--Maruyama (EM) iterates (with the same piecewise-constant control) is given by (\ref{Eeuler_maruyama}) with the continuous-time interpolated process $X^h(t)$. Then, from Theorem \ref{Thconvergence}, we have
\[
\bE\big[\,\sup_{0\leq t \leq T}|\tilde{X}^{h}(t) - X^h(t)|^2\,\big] \le \bE\big[\,\sup_{0\leq t \leq T}|X^{h}_t - X^h(t)|^2\,\big]\leq \hat{C}\,h.
\]
where $X^{h}$ is the solution of the (\ref{E1.1}) under the piece-wise constant policy $\tilde{U}^h$\,. Therefore, one can use either of the discrete-time Mckean-Vlasov systems to construct a piece-wise constant near-optimal policy for continuous time Mckean-Vlasov systems: The model above (\ref{DTApprSimul}) is more suitable for online or simulation based methods as in Q-learning implementations and the approximation via (\ref{Eeuler_maruyama}) is more suitable for off-line numerical methods. 
\end{remark}

%%%%%%%%%%%%%%%%%%%%%%%%%%%%%%%%%%
\subsubsection{Convergence for the finite horizon cost criterion}
\begin{theorem}\label{Tconvevaluefunc}
Suppose Assumptions \ref{Alipschitz}, \ref{Abounded}, and \ref{Acost} hold. Let $V^h(x)$ be the value function in the discrete-time model corresponding to the piece-wise constant policy (e.g., $m^{*\epsilon}$, see Lemma~\ref{LNearOptiA}). Then, we have 
\begin{equation*}
\lim_{h\downarrow 0}\abs{V_T^h(x) - V_T(x)} =  0 \quad\text{for all}\,\,\, x\in \Rd\,.
\end{equation*}
\end{theorem}

\begin{proof}
For any piece-wise constant policy $U^h\in \Uadm^h$, let $\{X_k^h\}$ be the discrete-time controlled process obtained from (\ref{Eeuler_maruyama}). Recall that $N_t^h = \sup \{ n : nh \leq t \}$. Then, we have
\begin{align*}
& \abs{\mathbb{E}_x^{U^h} \left[ \sum_{k=0}^{ N_T^h -1 } hc(X_k^h, \sL(X_k^h), U_k^h) + c_T(X^h_{N_T^h}, \sL(X_{N_T^h}^h))\right] \nonumber\\
&    - \mathbb{E}_x^{U^h} \left[ \int_{0}^{T} c(X^h(s), \sL(X^h(s)) U^h(s)) ds + c_T(X^h(T), \sL(X^h(T)))\right]}\nonumber\\ 
\leq & \abs{\mathbb{E}_x^{U^h} \left[\sum_{k=0}^{ N_T^h - 1 } \int_{kh}^{(k+1)h} \left(c(X_k^h, \sL(X_k^h), U_k^h) - c(X^h(s), \sL(X^h(s)) U^h(s))\right) ds \right]} \nonumber\\
&+ \abs{\mathbb{E}_x^{U^h} \left[ \int_{hN_T^h}^{T} c(X^h(s),\sL(X^h(s)), U^h(s)) ds \right]}\nonumber\\
=& \abs{\mathbb{E}_x^{U^h} \left[ \int_{hN_T^h}^{T} c(X^h(s),\sL(X^h(s)), U^h(s)) ds \right]}\nonumber\\
\leq & \|c\|_{\infty}h
\end{align*}

Thus, it follows that
\begin{equation*}
\mathcal{J}_{T,h}^{U^h}(x) = \mathbb{E}_x^{U^h} \left[ \int_{0}^{T}  c(X^h(s), U^h(s)) + c_T(X^h(T))\right] + \order(h).
\end{equation*}
%%%%%%%%%%%%%%%%%%%%%%%%%%%%%%%%%%%%
Since $m^{*\epsilon}$ is a piece-wise constant $\epsilon$-optimal control for the continuous time problem (existence is guaranteed by Lemma~\ref{LNearOptiA}), we have 
\begin{align}\label{bound22}
\cJ_T^{m^{*\epsilon}}(x) \leq V_T(x) + \epsilon.
\end{align}
 Consider the discrete-time model $X^h$ (here $h$ depends on $\epsilon$) associated with the piecewise constant control $m^{*\epsilon}$ (as in (\ref{Eeuler_maruyama})); with optimal value $V_T^h(x)$. Then, it follows that 
 \begin{align}\label{bound23}
 V_T^h(x) \leq \cJ_{T,h}^{m^{*\epsilon}}(x).
 \end{align}
Since $c, c_T$ are Lipschitz continuous, as we have derived the estimate \eqref{EcostConver1A}, for some constant $C_5 > 0$ we get
\begin{align*}
\abs{\cJ_T^{m^{*\epsilon}}(x) - \cJ_{T,h}^{m^{*\epsilon}}(x)} \leq & \abs{\mathbb{E}_{x} \left[ \int_0^T c(X_t, \sL(X_t), m^{*\epsilon}_t) dt + c_T(X_T, \sL(X_T)) \right] \nonumber\\
&- \mathbb{E}_{x} \left[ \int_0^T c(X^h(t), \sL(X^h(t)), m^{*\epsilon}_t) dt + c_T(X^h(T), \sL(X^h(T))) \right]} + \abs{\order(h)}\nonumber\\
\leq & C_5\sqrt{\mathbb{E}_x[\sup_{t\leq T}\abs{X_t - X^h(t)}^2]} + \|c\|_{\infty} h
\end{align*}

From Corollary~\ref{cor:EMAppro1AMV}, we know that 
$\mathbb{E}[\sup_{t\leq T}\abs{X_t - X^h(t)}^2] \leq \bar{C}_{EM} h\,.$ This implies that
\begin{align}\label{ETconvevaluefunc1A}
\abs{\cJ_T^{m^{*\epsilon}}(x) - \cJ_{T,h}^{m^{*\epsilon}}(x)} \leq (C_5\sqrt{\bar{C}_{EM}} + \|c\|_{\infty})\sqrt{h} = \hat{C}_5\sqrt{h}, 
\end{align}
where the constant $\hat{C}_5 := (C_5\sqrt{\bar{C}_{EM}} + \|c\|_{\infty})$ is independent of the piece-wise constant policy $m^{*\epsilon}$\,.
Thus, from (\ref{bound22}, \ref{bound23}) and (\ref{ETconvevaluefunc1A}) we get
\begin{equation}\label{EconvevaluefuncA}
V_T^h(x) \leq \cJ_{T,h}^{m^{*\epsilon}}(x) \leq \cJ_{T}^{m^{*\epsilon}}(x) + \hat{C}_5\sqrt{h} \leq V_T(x) + \epsilon +\hat{C}_5\sqrt{h}\,.
\end{equation}

For the lower bound, let $U^{h,*}\in \Uadm^h$ be an optimal control of the discretized system\,. Viewing it as an element of $\cM(T)$, in view of (\ref{ETconvevaluefunc1A}), we have
\begin{equation}\label{EconvevaluefuncB}
V_T^h(x) = \cJ_{T,h}^{U^{h, *}}(x) \geq \cJ_{T}^{U^{h,*}}(x) - \hat{C}_5\sqrt{h} \geq  V_T(x) - \hat{C}_5\sqrt{h}\,.
\end{equation} 

Since $\epsilon$ is arbitrary, from (\ref{EconvevaluefuncA}) and (\ref{EconvevaluefuncB}) we obtain the desired result\,. This completes the proof\,.
\end{proof}
\begin{corollary}\label{cor:CNearFinite}
From the proof of the above theorem (as we have obtained the estimate (\ref{ETconvevaluefunc1A})), it follows that for piecewise-constant policies (aligned with the Euler-Maruyama grid) there is a uniform error estimate
\[
\sup_{m^{h}\ \text{piecewise-const on grid }h}\big|\cJ_T^{m^{h}}(x) - \cJ_{T,h}^{m^{h}}(x)\big| \le \tilde{C}\,h^{\frac{1}{2}},
\]
for some constant $\tilde{C}$ that depends only on the model parameters and $T$.
\end{corollary}
Next, using the above continuity result of the value functions with respect to the discrete-time approximation, we establish the near optimality of the discrete-time optimal policy in continuous-time Mckean-Vlasov system.
%%%%%%%%%%%%
\begin{remark}
In the classical (non–mean-field) controlled diffusion setting, a substantial body of work establishes \emph{explicit convergence rates} for discrete-time or finite-difference approximations by exploiting the regularity of the solutions of the associated Bellman equations.  In particular, Krylov \cite{KN98A,KN2000A} proved stability and derived error estimates for finite-difference schemes applied to degenerate Bellman equations, based on delicate mean-value theorems for stochastic integrals \cite{KN2001AA}. Moreover, for uniformly non-degenerate controlled diffusions, Krylov showed in \cite{KN99AA} that policies which are constant on intervals of length $h^2$ approximate the continuous-time value function with error of order $h^{1/3}$.  These rates were subsequently improved by Barles, Jakobsen, and co-authors \cite{BR-02,BJ-06} using refined viscosity-solution techniques.  In contrast, for McKean--Vlasov control problems the Bellman or master equations generally lack the regularity required for such PDE-based estimates.  Consequently, our approach relies instead on \emph{probabilistic arguments} and compactness of controlled laws: these yield robust convergence of discrete-time value functions to their continuous-time counterparts, with an explicit rate of convergence under some mild structural assumptions over the piece-wise constant control policy space (see, Corollary~\ref{cor:EMAppro1AMV} and Corollary~\ref{cor:CNearFinite}).
\end{remark}
%%%%%%%%%%%%
\begin{theorem}\label{TNearFinite}
Suppose Assumptions \ref{Alipschitz}, \ref{Abounded}, and \ref{Acost} hold. Let $U^{*,h}\in \Um^h$ (with continuous time interpolation $\tilde{U}^{*,h}$), be an optimal policy of the discrete-time approximating model (\ref{Eeuler_maruyama}) (obtained in Theorem~\ref{TOptiDiscMcKV1A}). Then, we have
\begin{equation*}
\lim_{h\downarrow 0}\abs{\cJ_{T}^{\tilde{U}^{*,h}}(x) - V_T(x)} = 0  \quad\text{for all}\,\,\, x\in \Rd\,.
\end{equation*}
\end{theorem}
\begin{proof}
By triangle-inequality for each $x\in\Rd$, (since $V_T^{h}(x) = \cJ_{T,h}^{U^{*, h}}(x)$) we have
\begin{equation}
|\cJ_{T}^{\tilde{U}^{*, h}}(x) - V_T(x)| \leq |\cJ_{T}^{\tilde{U}^{*, h}}(x) - \cJ_{T,h}^{U^{*, h}}(x)| + |V_T^{h}(x) - V_T(x)|
\end{equation} From Theorem~\ref{Tconvevaluefunc} we know that $\lim_{h\downarrow 0}|V_T^{h}(x) - V_T(x)| = 0$\,. Also, arguing as in  the proof of the Theorem~\ref{Tconvevaluefunc} (in particular see Corollary~\ref{cor:CNearFinite}), we have $|\cJ_{T}^{\tilde{U}^{*, h}}(x) - \cJ_{T,h}^{U^{*, h}}(x)| \leq \tilde{C}\,\sqrt{h}$\,. This completes the proof. 
\end{proof}
%%%%%%%%%%%%%%%%%%%%%%%%
\subsubsection{Convergence for the discounted cost criterion}
In the next theorem, we show that as the parameter of discretization approaches to zero, the optimal value of the discretized control problem converges to continuous time optimal value.
\begin{theorem}\label{TConvergenceConti}
Suppose that Assumptions \ref{Alipschitz}, \ref{Abounded}, and \ref{Acost} hold. Then, we have
\begin{equation}\label{ETConvergenceConti}
\lim_{h\downarrow 0}\abs{\cJ_{\alpha,h}^{*}(x, c) - V_{\alpha}(x)} = 0 \quad\text{for all}\,\,\, x\in \Rd\,.
\end{equation}
\end{theorem}
\begin{proof}
\textbf{Step-1:}
Since $\beta = e^{-\alpha h}$ for $U^h\in \Uadm^h$, with $v^h(\cdot)$ be the associated continuous-time interpolated policy (defined as in (\ref{EInterPol})), it follows that 
\begin{align*}
&\bigg|\Exp_x^{U^h} \left[\sum_{k=0}^{\infty} \beta^k c_h(X_k^h, U_k^h)\right] - \Exp_x^{U^{h}} \left[\int_{0}^{\infty} e^{-\alpha s} c(X^{h}(s), v^{h}(s)) \right] \bigg| \nonumber\\
& = \bigg|\Exp_x^{U^h} \left[\sum_{k=0}^{\infty} \int_{kh}^{(k+1)h}\left(\beta^k c(X_k^h, U_k^h) - e^{-\alpha s} c(X^{h}(s), v^{h}(s)) \right)ds\right]\bigg| \nonumber\\ 
&\leq \norm{c}_{\infty}\sum_{k=0}^{\infty} \int_{kh}^{(k+1)h} \abs{\left(e^{-\alpha hk} - e^{-\alpha s} \right)} ds \nonumber\\
& \leq \norm{c}_{\infty}h\left(1 - e^{-\alpha h} \right)\sum_{k=0}^{\infty} e^{-\alpha hk} = \norm{c}_{\infty}h\,.
\end{align*}
This implies that
\begin{equation}\label{EConvergenceContiA}
\cJ_{\alpha,h}^{U^{h}}(x) \,=\, \Exp_x^{U^{h}} \left[\int_{0}^{\infty} e^{-\alpha s} c(X^{h}(s), v^{h}(s)) \right] + \order(h)\,.
\end{equation}
\textbf{Step-2:}
Let $U^{h,*} \in \cM(\infty)$ be the continuous-time interpolated policy of an optimal policy for the discrete-time control problem (existence follows from a similar argument as in\cite{sanjarisaldiy2024optimality} \cite{bauerle2023mean}). Then arguing as in the proof of Theorem~\ref{Tconvevaluefunc}, and using (\ref{EConvergenceConti1A}) and (\ref{EConvergenceContiA}) for all $x\in \Rd$ and $\epsilon > 0$ it follows that 
\begin{align}\label{EConvergenceContiAA}
&\abs{\cJ_{\alpha}^{U^{h,*}}(x) - \cJ_{\alpha, h}^{U^{h,*}}(x)} \nonumber\\
&\leq \abs{\mathbb{E}_{x} \left[ \int_0^{T_{\epsilon}} e^{-\alpha t}\left(c(X_t^h, \sL(X_t^h),U^{h,*}(t)) -c(X^h(t), \sL(X^h(t)), U^{h,*}(t))\right) dt\right]} + \frac{\epsilon}{2}\nonumber\\   
&\leq  \hat{C}_5\sqrt{h} + \frac{\epsilon}{2}\,,
\end{align} where $X_t^h$ is the solution of the (\ref{E1.1}) under policy $U^{h,*}$ and $X^h(t)$ is the continuous-time interpolated process of the corresponding discrete-time approximating Mckean-Vlasov processes (obtained by Euler-Maruyama scheme (\ref{Eeuler_maruyama})). 

This implies
\begin{equation}\label{EConvergenceContiB}
V_{\alpha}^h = \cJ_{\alpha, h}^{U^{h,*}}(x) \geq \cJ_{\alpha}^{U^{h,*}}(x) - \hat{C}_5\sqrt{h} - \frac{\epsilon}{2}\geq V_{\alpha}(x) - \hat{C}_5\sqrt{h} - \frac{\epsilon}{2}\,.
\end{equation}

Again, for any $\epsilon > 0$ since $m^{*\epsilon}\in \cM(\infty)$ is a piece-wise constant $\frac{\epsilon}{2}$-optimal control of the continuous time problem, following the steps as in the proof of Theorem~\ref{Tconvevaluefunc}, we deduce that
\begin{equation}\label{EConvergenceContiC}
V_{\alpha}^h(x) \leq \cJ_{\alpha,h}^{m^{*\epsilon}}(x) \leq \cJ_{\alpha}^{m^{*\epsilon}}(x) + \hat{C}_5\sqrt{h} + \frac{\epsilon}{2}\leq V_{\alpha}(x) + \epsilon +\hat{C}_5\sqrt{h}\,.
\end{equation}
Since $\epsilon$ is abitrary, from (\ref{EConvergenceContiB}) and (\ref{EConvergenceContiC}), we get (\ref{ETConvergenceConti})\,. This completes the proof\,.
\end{proof}
The following theorem proves the near optimality of the discrete-time optimal policy in the continuous time model as the parameter of the discretization appraches to zero\,.
\begin{theorem}\label{TNeardiscounted}
Suppose that Assumptions \ref{Alipschitz}, \ref{Abounded}, and \ref{Acost} hold. Then, we have
\begin{equation}
\lim_{h\downarrow 0}\abs{\cJ_{\alpha}^{U^{h,*}}(x) - V_{\alpha}(x)} = 0 \quad\text{for all}\,\,\, x\in \Rd\,.
\end{equation}
\end{theorem}
\begin{proof} By triangle-inequality for each $x\in\Rd$, we have
\begin{equation}\label{Eneardiscounted1A}
|\cJ_{\alpha}^{U^{h,*}}(x) - V_{\alpha}(x)| \leq |\cJ_{\alpha}^{U^{h,*}}(x) - V_{\alpha}^{h}(x)| + |V_{\alpha}^{h}(x) - V_{\alpha}(x)|
\end{equation} From Theorem~\ref{TConvergenceConti} we have $\lim_{h\downarrow 0}|V_{\alpha}^{h}(x) - V_{\alpha}(x)| = 0$\,. Since $\cJ_{\alpha,h}^{U^{h,*}}(x) = V_{\alpha}^{h}(x)$ arguing as in the proof Theorem~\ref{TConvergenceConti} (see the estimate (\ref{EConvergenceContiAA})), we obtain $|\cJ_{\alpha}^{U^{h,*}}(x) - \cJ_{\alpha,h}^{U^{h,*}}(x)| \leq \hat{C}_5\sqrt{h} + \frac{\epsilon}{2}$\,. Since $\epsilon$ is arbitrary, from \eqref{Eneardiscounted1A} we obtain our result. 
\end{proof}
%%%%%%%%%%%%%%%%%%%%%%%%

%%%%%%%%%%%%%%%%%%%%%%%%%%%%%%%
\section{Near-Optimality of the Discrete-Time Approximate McKean-Vlasov Solution for $N$-Particle Interacting Diffusions}\label{sec:Nparticle}

In the paper so far, we established near optimality of time-discretized McKean-Vlasov solution for the original diffusion. In this section, we show that if a collection of interacting agents adopt the policy obtained from the discrete-time approximation of the McKean-Vlasov problem, this policy becomes near optimal under centralized information among all policies (and due to the relaxed information requirements of the policy obtained from the McKean-Vlasov solution, under all information structures where each agent only has access to local state and mean-field information). This is a computationally very important result as it is well known that solving problems with many agents becomes computationally intractactable.  

As noted earlier, it has been established that under various conditions the solution of the $N$-agent interacting particle system with centralized information structure admits an optimal solution which converges to the McKean-Vlasov solution as $N \to \infty$, see \cite{budhiraja2012large,fornasier2014mean,fornasier2019mean,lacker2017limit} in continuous-time under complementary conditions and setups (including both deterministic and stochastic settings), and \cite{sanjarisaldiy2024optimality} in discrete-time. In this section, we bridge these findings in the context of our paper and develop asymptotic approximation results both in time and in the number of interacting agents. 

Consider a system of $N$ interacting particles (agents) evolving according to the controlled McKean–Vlasov dynamics.  Specifically, each particle $i=1,\dots,N$ has state $X^{N,i}_t$ satisfying the SDE
\begin{equation}\label{ENState}
dX^{N,i}_t \;=\; b\bigl(X^{N,i}_t,\mu^N_t,U^{N,i}_t\bigr)\,dt \;+\;\sigma\bigl(X^{N,i}_t,\mu^N_t\bigr)\,dW^i_t,\quad 
X^{N,i}_0 = x,
\end{equation} 
where $\{W^i\}$ are independent Brownian motions, $U^{N,i}_t$ is the control for particle $i$, and 
\[
\mu^N_t \;=\; \frac1N\sum_{j=1}^N \delta_{X^{N,j}_t}
\]
is the empirical distribution of the particle states.  
%%%%%%%%%%%%%%%%%%%%%%%%%%%%%%%%%%%%%

For the N-particle system, we define the space of admissible control policies by
\begin{align*}
&\cM^N(T) = \bigg\{{\bf m}^N = (m^{N,1}, m^{N,2}, \dots, m^{N,N})\mid m^{N,i}: (\Omega, \sF, \sP)\to \fM(T)\,\,\text{for each}\,i\,\,\text{and}\,\, \\
&{\bf m}^N_{[0, s]} \,\,\text{is independent of}\,\, {\bf W}_t - {\bf W}_s\,\, \text{for any} \,\, 0 \leq s\leq t \leq T \,\,\text{where}\,\,{\bf W} = (W^1, W^2, \dots, W^N)\bigg\}\,.
\end{align*}

Since the range space of each coordinate $\fM(T)$ is compact, any sequence in $\cM^N(T)$ is tight\,. Moreover, by standard chattering lemma, as in Lemma~\ref{LChatt}, we have piecewise-constant non-anticipative control policies are dense in $\cM^N(T)$\,.

%%%%%%%%%%%%%%%%%%%%%%%%%%%%%%%%%%%%
We assume the same Lipschitz and boundedness conditions (Assumptions \ref{Alipschitz}, \ref{Abounded}, and \ref{Acost}) on $b,\sigma,c,c_T$ as in Section~\ref{Scontinuous}.  The finite-horizon cost for the $N$-particle system (average per agent) is naturally defined in analogy to (\ref{EFcost}) as 
\[
\cJ^N_T(x,U^N) \;=\; \frac{1}{N} \sum_{i=1}^N \mathbb{E}_{x}^{U^{N,i}}\Bigl[\int_0^T c\bigl(X^{N,i}_t,\mu^N_t,U^{N,i}_t\bigr)\,dt \;+\; c_T\bigl(X^{N,i}_T,\mu^N_T\bigr)\Bigr],
\] 
where ${\bf U}^N=(U^{N,1},\dots,U^{N,N})$ is the collection of controls.  (In particular, for a symmetric policy one has $\cJ^N_T(x,{\bf U}^N)= \mathbb{E}\bigl[\int_0^T c(X^{N,1}_t,\mu^N_t,U^{N,1}_t)\,dt+c_T(X^{N,1}_T,\mu^N_T)\bigr]$.)  The optimal value for the continuous-time $N$-agent problem is defined as 
\[
V_T^N(x) := \inf_{{\bf U}^N\in\cM^N (T)}\cJ^N_T(x,U^N)\,.
\]
Under further technical assumptions on the system model \cite{lacker2017limit}, it is known that as $N\to\infty$ this $N$-particle system converges (in law) to the McKean–Vlasov system of Section~2, and its optimal cost $V_T^N(x)$ approaches the McKean–Vlasov value $V_T(x)$. Accordingly, one can arrive at an approximation in both discrete-time and finite number of agents via first approximations in the number of agents \cite{budhiraja2012large,fornasier2019mean,lacker2017limit} and then discrete-time approximation for centralized diffusions paper as in  \cite{pradhanyuksel2023DTApprx}; or via the approach of our current paper to approximate with a discrete-time McKean-Vlasov model and then via \cite{sanjarisaldiy2024optimality} to relate the problem to one with a finite number of agents with dynamics as given in (\ref{ENState}) below. While these are consistent in leading to the near approximability in both time discretization and $N$-particle approximation, the technical conditions are complementary. For example we do not need the demanding convexity condition in \cite[Assumption C]{lacker2017limit} for the near optimality of Markov (in state and mean-field) policies; we also note that \cite[Assumption C]{lacker2017limit} is not needed if one is only to focus on the convergence of the joint path and control measures to those which are optimal \cite{{lacker2017limit}}).

The associated $\alpha$-discounted cost for the $N$-interacting particle system is given by
\[
\cJ^N_{\alpha}(x,{\bf U}^N) \;=\; \frac{1}{N} \sum_{i=1}^N \mathbb{E}_{x}^{U^{N,i}}\Bigl[\int_0^{\infty} e^{\alpha t} c\bigl(X^{N,i}_t,\mu^N_t,U^{N,i}_t\bigr)\,dt \Bigr],
\] and the optimal value is defined as
\[
V_{\alpha}^N(x) := \inf_{{\bf U}^N\in\cM^N (T)}\cJ^N_{\alpha}(x,{\bf U}^N)\,.
\]

We now introduce a discrete-time approximation for the $N$-particle system and show that the optimal discrete-time policy for the McKean–Vlasov problem remains near-optimal for the finite-$N$ system.  As in Section~\ref{SdiscreteSDE}, fix a time-step $h>0$ and let $t_k=kh$.  We use the following Euler–Maruyama scheme for each particle: for $i=1,\dots,N$,
\begin{equation}\label{ENStateDiscrete}
X^{N,h,i}_{k+1} \;=\; X^{N,h,i}_k \;+\; b\bigl(X^{N,h,i}_k,\mu^{N,h}_k,U^{N,i,h}_k\bigr)\,h \;+\; \sigma\bigl(X^{N,h,i}_k,\mu^{N,h}_k\bigr)\,\Delta W^i_k,\quad X^{N,i}_0=x,
\end{equation}
where $\Delta W^i_k = W^i_{t_{k+1}}-W^i_{t_k}$ and 
\[
\mu^N_k \;=\; \frac1N\sum_{j=1}^N \delta_{X^{N,j}_k}\,.
\]
For any $1\leq i \leq N$, let $X^{N,h,i}(\cdot)$ be the continuous-time interpolated process of the discrete-time process $\{X^{N,h,i}_{k}\}_k$\,. This is the natural $N$-particle extension of the single-agent discrete scheme (\ref{Eeuler_maruyama}), and $\mu^N_k$ is the empirical measure analogous to $\hat\mu^h_k$ in Remark~\ref{ER4.1}.  Under a given discrete-time admissible control policy ${\bf U}^{N,h}_k =(U^{N,1,h}_k,\dots,U^{N,N,h}_k)$, the $N$-agent discrete-time finite-horizon cost (average per agent) is defined by
\[
\cJ^N_{T,h}(x,{\bf U}^{N,h}) \;=\; \frac{1}{N}\sum_{i=1}^N\mathbb{E}_{x}^{U^{N,i,h}}\Bigl[\sum_{k=0}^{N_h^T-1} c\bigl(X^{N,i}_k,\mu^N_k,U^{N,i}_k\bigr)\,h \;+\; c_T\bigl(X^{N,i}_{N_h^T},\mu^N_{N_h^T}\bigr)\Bigr],
\]
where $N_h^T=\sup\{k:kh\le T\}$. Note that this is exactly the particle-averaged analogue of the single-agent cost $\cJ^{U^h}_{T,h}(x)$ given in (\ref{Ediscrete_cost}). As we will be importing the policy obtained from the McKean-Vlasov limit, we take the admissible discrete policy $\{U^N_k\}$ to be symmetric across agents (which will be shown further below to be near optimal under any centralized information structure). Let $V_T^{N,h}$ denote the optimal value of the discrete-time N-particle system. Also, for the associated continuous-time interpolated process $X^{N,i,h}(\cdot)$, the cost under any admissible policy is defined as follows: 
\begin{equation}\label{ENcostinterpo1A}
\bar{\cJ}^N_{T,h}(x,{\bf U}^{N,h}) \;=\; \frac{1}{N}\sum_{i=1}^N\mathbb{E}_{x}^{U^{N,i}}\Bigl[\int_{0}^{T} c\bigl(X^{N,i,h}(s),\mu^{N,h}_s,U^{N,i,h}(s)\bigr)\, ds \;+\; c_T\bigl(X^{N,i,h}(T),\mu^{N,h}_{T}\bigr)\Bigr]\,.
\end{equation}

Let $U^{h,*}$ denote an optimal discrete-time control for the McKean–Vlasov model as in (\ref{Ediscrete_optimal_control}).  We extend $U^{h,*}$ to the $N$-agent system by using the same (Markov for finite horizon or stationary Markov for discounted infinite horizon) policy for each particle, i.e.\ 
\begin{align}\label{SymmetricMcKeanSoln}
U^{N,i}_k = U^*_h\bigl(k,X^{N,i}_k,\mu^N_k\bigr),
\end{align}
for all $i$  (as obtained in Theorem~\ref{TOptiDiscMcKV1A}). We will show that this policy achieves an $N$-agent cost $\cJ^N_{T,h}(x,U^N)$ that is arbitrarily close to the optimal value of the $N$-particle system as $h\to0$ and $N$ large enough.  

Using Theorem \ref{TNearFinite} for finite horizons, and Theorem \ref{TNeardiscounted} for discounted infinite horizons, we have established near optimality of policies obtained from a time-discretized McKean-Vlasov dynamics. Building on \cite{sanjarisaldiy2024optimality}, we have then that the finite agent approximation (\ref{ENStateDiscrete}) is such that the optimal cost for this finite-agent approximation is close to the original McKean-Vlasov problem. 

In this section, we show that the discrete-time approximation for the McKean-Vlasov problem and its solution obtained via Theorem~\ref{TOptiDiscMcKV1A} is near optimal for the finite agent continuous-time problem for the dynamics given in (\ref{ENState}).

Now as in the Theorem~\ref{Thconvergence}, we have the following approximation for the N-particle system\,.
\begin{theorem} \label{ThconvergenceNparti}
Suppose Assumptions \ref{Alipschitz} and \ref{Abounded} hold. Let $X_t^N$ be the solution of the controlled $N$-particle system (\ref{ENState}) under ${\bf m}^N\in \cM^N(T)$ and $X_{k}^{N,i,h}$ (with continuous time interpolated process $X^{N,i,h}(\cdot)$) be the solution of the Euler-Maruyama scheme (\ref{ENStateDiscrete}) under a piece-wise constant non-anticipative policy ${\bf m}^{N,h}$ (which approximates ${\bf m}^{N}$ weakly a.s.). Then, we have
\begin{equation}\label{ENConv1A}
 \lim_{h\downarrow 0}\sup_{1\leq i \leq N}\mathbb{E}[\sup_{t \in [0, T]}|X_t^{N,i} - X^{N,i,h}(t)|^2] = 0\,.
\end{equation}
\end{theorem}
\begin{proof} 
Using the fact that $b$, $\upsigma$ are bounded and the Burkholder-Davis-Gundy (BDG) inequality, as in the proof of Theorem~\ref{Thconvergence}, we have
\begin{align}\label{ENparticleThconv1A}
&\mathbb{E}[\sup_{t\leq T}\abs{X_t^{N,i} - X^{N,i,h}(t)}^2] \nonumber\\
\leq & 4\Big(\bar{C}_1 h + \bar{\eta}_{h} + T\mathbb{E}[\int_{0}^{T} \abs{b(X_t^{N,i}, \mu^{N}_s, m^{N,i,h}(s)) ds - b(X^{N,i,h}(s), \mu^{N,h}_s, m^{N,i,h}(s))}^2 ds] \nonumber\\
& +  \bar{C}_2\mathbb{E}[\int_{0}^{T} \abs{\left(\sigma(X^{N,i}_s, \mu^{N}_s) - \sigma(X^{N,i,h}(s), \mu^{N,h}_s)\right)}^2 d s \Big)\nonumber\\
\leq & 4\Big(\bar{C}_1 h + \bar{\eta}_{h} + 2(T + \bar{C}_2)\mathbb{E}[\int_{0}^{T} \abs{X_s^{N,i} - X^{N,i,h}(s)}^2 + W_1(\mu^{N}_s, \mu^{N,h}_s))^2 ds]\Big),  
\end{align} where $\bar{C}_1 > 0$ is a constant which depends only on the bounds of $b, \upsigma$,\, the constant $\bar{C}_2 >0$ is obtained from the BDG inequality and
\begin{align*}
\bar{\eta}_{h} = \max_{1\leq i\leq N}\mathbb{E}[\sup_{0\leq t \leq T}\abs{\int_{0}^{t}b(X_t^{N,i}, \mu^{N}_s, m^{N,i}_s) ds - \int_{0}^{t}b(X_t^{N,i}, \mu^{N}_s, m^{N,i,h}(s)) ds}^2]\,.
\end{align*}
Also, we have $W_1(\mu^{N}_s, \mu^{N,h}_s)^2 \leq W_2(\mu^{N}_s, \mu^{N,h}_s)^2 \leq \frac{1}{N}\sum_{j=1}^N |X^{N,j}_s - X^{N,j,h}_s|^2$. In view of this, we obtain
\begin{align*}
\max_{1\leq i \leq N}\mathbb{E}[\sup_{t\leq T}\abs{X_t^{N,i} &- X^{N,h,i}(t)}^2] \nonumber \\
&\leq 4\Big(\bar{C}_1 h + \bar{\eta}_{h} + 4(T + \bar{C}_2)\int_{0}^{T}\max_{1\leq i \leq N}\mathbb{E}[\sup_{s\leq t}\abs{X_s^{N,i} - X^{N,h,i}(s)}^2] ds \Big).
\end{align*}
Now applying the Gronwall's inequality, we deduce that
\begin{equation}\label{EThconvN1B}
\max_{1\leq i \leq N}\mathbb{E}[\sup_{t\leq T}\abs{X_t^{N,i} - X^{N,h,i}(t)}^2] \leq 4(\bar{C}_1 h + T\bar{\eta}_{h})e^{16T(T + \bar{C}_2)}
\end{equation}
Since ${\bf m}^{N,h}(\cdot) \to {\bf m}^{N}$ weakly (a.s. in $\Omega$), by dominated convergence theorem (as in (\ref{EThExistenceC1A})) we get $\bar{\eta}_h \to 0$ as $h\downarrow 0$\,. Thus, by letting $h\downarrow 0$ from (\ref{EThconvN1B}), we deduce that (\ref{ENConv1A}) holds. This completes the proof of the theorem\,.
\end{proof}
In view of the above theorem, under any fixed piecewise constant policy ${\bf m}^{N,h}$ (with step size $h$) we have the following strong uniform approximation bound. In the proof of the above theorem if we fix a piece-wise constant policy then $\bar{\eta}_{h} = 0$\,.
\begin{corollary}\label{cor:EMAppro1A}
Suppose Assumptions \ref{Alipschitz} and \ref{Abounded} hold. Let $X_t^{N,i}$ be the solution of the controlled $N$-particle system (\ref{ENState}) under a piece-wise constant non-anticipative policy ${\bf m}^{N,h}$ and $X_{k}^{N,i,h}$ (with continuous time interpolated process $X^{N,i,h}(\cdot)$) be the solution of the Euler-Maruyama scheme (\ref{ENStateDiscrete}) under the same piece-wise constant policy. Then, for some positive constant $C_{EM}$ (independent of $N$ and ${\bf m}^{N,h}$) we have
\begin{equation}\label{ENConvCor1A}
\sup_{1\leq i \leq N}\mathbb{E}[\sup_{t \in [0, T]}|X_t^{N,i} - X^{N,i,h}(t)|^2] \leq C_{EM}h\,.
\end{equation}    
\end{corollary}

Next lemma provides uniform  discrete-time approximation bounds under piece-wise constant control policies\,.
\begin{lemma}\label{lem:cost_consistency}
Suppose Assumptions \ref{Alipschitz}, \ref{Abounded} and \ref{Acost} hold. Then, for any piecewise-constant non-anticipative (with uniform grid size $h$) policy \({\bf m}^{N,h}\), we have
\[
\big|\cJ_T^N(x, {\bf m}^{N,h}) - \cJ^{N}_{T,h}(x, {\bf m}^{N,h})\big| \le C_{\mathrm{cost}}\, \sqrt{h},
\]
with \(C_{\mathrm{cost}}\) depending only on \(T,\) and the Lipschitz constants of $b, \upsigma, c, c_T$; in particular \(C_{\mathrm{cost}}\) can be chosen independent of \(N\) and of \({\bf m}^{N,h}\).
\end{lemma}

\begin{proof}
By the triangle inequality, we have 
\begin{align*}
&\big|\cJ_T^N(x, {\bf m}^{N,h}) - \cJ^{N}_{T,h}(x, {\bf m}^{N,h})\big| \nonumber \\
&\le \big|\cJ^{N}_{T,h}(x, {\bf m}^{N,h}) - \bar{\cJ}^{N}_{T,h}(x, {\bf m}^{N,h})\big| + \big|\bar{\cJ}^{N}_{T,h}(x, {\bf m}^{N,h}) - \cJ_T^N(x, {\bf m}^{N,h})\big|   
\end{align*}
Arguing as in the proof of the Theorem~\ref{Tconvevaluefunc}, we have 
\begin{align}\label{Ecost_consistency1A}
\big|\cJ^{N}_{T,h}(x, {\bf m}^{N,h}) - \bar{\cJ}^{N}_{T,h}(x, {\bf m}^{N,h})\big| \leq \|c\|_{\infty}h    
\end{align}
Since $c, c_T$ are Lipschitz continuous, for some constant $\bar{C}_5 > 0$ we get
\begin{align}\label{Ecost_consistency1B}
&\abs{\bar{\cJ}^{N}_{T,h}(x, {\bf m}^{N,h}) - \cJ_T^N(x, {\bf m}^{N,h})} \nonumber \\
\leq & \frac{1}{N}\sum_{i=1}^{N} \abs{\mathbb{E}_{x} \left[ \int_0^T c(X^{N,i}_t, \mu^{N}_t, {\bf m}^{N,i,h}(t)) dt + c_T(X_T^{N,h}, \mu^{N}_T) \right] \nonumber\\
& - \mathbb{E}_{x} \left[ \int_0^T c(X^{N,i,h}(t), \mu^{N,h}_t, {\bf m}^{N,i,h}(t)) dt + c_T(X^{N,i,h}(T), \mu^{N,h}_T) \right]}\nonumber\\
\leq & \bar{C}_5\sqrt{\mathbb{E}_x[\max_{1\leq i\leq N}\sup_{t\leq T}\abs{X^{N,i}_t - X^{N,i,h}(t)}^2]} \leq \bar{C}_6\sqrt{h},
\end{align} the last inequality follows eq. (\ref{ENConvCor1A})\,. Combining (\ref{Ecost_consistency1A}) and (\ref{Ecost_consistency1B}) we obtain our result.
\end{proof}
\begin{corollary}\label{Cor:forNparticle1a}
In view of the above result, it follows that for piecewise-constant policies (aligned with the Euler-Maruyama grid) there is a uniform error estimate
\[
\sup_{{\bf m}^{N,h}\ \text{piecewise-const on grid }h}\big|\cJ_T^N(x,{\bf m}^{N,h}) - \cJ_{T,h}^{N}(x,{\bf m}^{N,h})\big| \le \bar{C}\,h^{\frac{1}{2}},
\]
for a constant $\bar{C}$ depending only on model parameters and $T$.
\end{corollary}

Arguing as in the Theorem~\ref{Tconvevaluefunc}, in the following theorem we show that the optimal value of the discrete-time N-particle system converges to the optimal value  of the N-particle continuous-time system\,.
\begin{theorem}\label{TNconvevaluefunc}
Suppose Assumptions \ref{Alipschitz}, \ref{Abounded}, and \ref{Acost} hold. Then, for all $x\in \Rd$ we have
\begin{equation*}
\lim_{h\downarrow 0}\abs{V_T^{N,h}(x) - V_T^N(x)} = 0\,.
\end{equation*}
\end{theorem}
\begin{proof}
Following the proof techniques as in the Theorem~\ref{ThExistenceA}, we have ${\bf m}^N \mapsto \cJ_T^{N}(x, {\bf m}^N)$ is continuous on $\cM^N(T)$\,. Also, applying the chattering lemma, it folows that for any $\epsilon > 0$ there exists a piece-wise constant non-anticipative control ${\bf m}^{N,\epsilon}$ satisfying 
\begin{equation}
\cJ_T^N(x, {\bf m}^{N,\epsilon}) \leq  V_T^N(x) + {\epsilon}\,.     
\end{equation}

From Lemma~\ref{lem:cost_consistency}, for small enough $h$, we have
\begin{equation}
\big|\cJ_T^N(x, {\bf m}^{N,\epsilon}) - \cJ^{N}_{T,h}(x, {\bf m}^{N,\epsilon})\big| \le C_{\mathrm{cost}}\, \sqrt{h},
\end{equation} thus 
\begin{align}\label{ETNconvevaluefunc1A}
V_T^{N,h}(x) \le \cJ_{T,h}^{N}(x, {\bf m}^{N,\epsilon}) \le \cJ_T^N(x, {\bf m}^{N,\epsilon}) + C_{cost} h^{1/2} \le V_T^N(x) + \varepsilon + C_{cost} h^{1/2}.
\end{align}

For each $h>0$ pick an (near) optimal discrete policy ${\bf m}^{N,\epsilon, h}$ satisfying $V_T^{N,h}(x) + \epsilon \geq \cJ_{T,h}^{N}(x, {\bf m}^{\epsilon, h})$. Extend ${\bf m}^{N,\epsilon, h}$ to a relaxed control in $\cM^N(T)$ by holding actions constant on each grid of size $h$; this yields a sequence $\{\widetilde{{\bf m}}^{N,\epsilon, h}\}\subset \cM^N(T)$. Thus, we have
\begin{equation}\label{ETNconvevaluefunc1B}
V_T^{N,h}(x) + \epsilon \ge \cJ_{T,h}^{N}(x, {\bf m}^{N,\epsilon, h}) \ge \cJ_T^N(x, \widetilde{{\bf m}}^{N,\epsilon, h}) - C_{cost} h^{1/2} \ge V_T^N(x) - C_{cost} h^{1/2}\,.    
\end{equation}
Since $\epsilon$ is  arbitrary combining (\ref{ETNconvevaluefunc1A}) and (\ref{ETNconvevaluefunc1B}), we obtain the desired result\,.
\end{proof}
Now we are ready to derive the main result of this sub-section. Let ${\bf U}^{N,h} = (U^h,\dots, U^h)$ where $U^h$ is an optimal solution for the discrete-time McKean-Vlasov problem adopted by each agent (recall (\ref{SymmetricMcKeanSoln})). 

\begin{theorem}\label{nearOptDiscFiniteHor}
Suppose Assumptions \ref{Alipschitz}, \ref{Abounded}, and \ref{Acost} hold. Then for any $\epsilon >0$, there exist $\hat{h}_{\epsilon}>0$ and $\hat{N}_{\epsilon}\in\mathbb{N}$ such that for all $0 < h < \hat{h}_{\epsilon}$ and all $N\geq \hat{N}_{\epsilon}$, we have
\[
\bigl|\cJ^N_{T}(x,{\bf U}^{N,h}) - V_T^N(x)\bigr| \;<\; \epsilon.
\]
\end{theorem}
\begin{proof}
By the triangle inequality we have
\begin{align*}
&\abs{\cJ^N_{T}(x,{\bf U}^{N,h})- V_T^N(x)} \nonumber \\
\leq & \abs{\cJ^N_{T}(x,{\bf U}^{N,h})- \cJ^N_{T,h}(x,{\bf U}^{N,h})} + \abs{\cJ^N_{T,h}(x,{\bf U}^{N,h})- V_T^{N,h}(x)} + \abs{V_T^{N,h}(x)- V_T^N(x)}\,.  
\end{align*}
From Lemma~\ref{lem:cost_consistency}, we have the first term of the above inequality goes to zero as $h\to 0$. Similarly, from Theorem~\ref{TNconvevaluefunc}, the third term of the above inequality goes to zero as $h\to 0$\,. Observe critically that these approximations hold uniformly in $N \in \mathbb{N}$. Therefore, for every $\epsilon > 0$ by taking $h$ small enough the errors are bounded away from $ 2\epsilon/3$. For the second term, the approach is then to take $N$ large enough to attain an error of at most $\epsilon / 3$. By arguing as in the proof of \cite[Theorem~4]{sanjarisaldiy2024optimality}, for any $\epsilon > 0$ there exists small enough $h_{\epsilon} >0$ and large enough $\hat{N}_{\epsilon}$ such that
$$\abs{\cJ^N_{T,h}(x,{\bf U}^{N,h})- V_T^{N,h}(x)} \leq \frac{\epsilon}{3}$$
for any $h\leq h_{\epsilon}$ and $N \geq \hat{N}_{\epsilon}$. This completes the proof.
\end{proof}
In view of the above near optimality result for the finite horizon cost and the estimate as in (\ref{EConvergenceConti1A}), we have the following near-optimality result for the discounted cost $N$-particle system. Let $U^h$ be the discounted optimal policy for the discrete-time McKean-Vlasov control problem and let ${\bf U}^{N,h} = (U^h,\dots, U^h)$ then we have the following.
\begin{theorem}\label{nearOptDiscDiscount}
Suppose Assumptions \ref{Alipschitz}, \ref{Abounded}, and \ref{Acost} hold. Then for any $\epsilon >0$, there exist $\hat{h}_{\epsilon}>0$ and $\hat{N}_{\epsilon}\in\mathbb{N}$ such that for all $0 < h < \hat{h}_{\epsilon}$ and all $N\geq \hat{N}_{\epsilon}$, we have
\[
\bigl|\cJ^N_{\alpha}(x,{\bf U}^{N,h}) - V_{\alpha}^N(x)\bigr| \;<\; \epsilon.
\]
\end{theorem}

\begin{remark}[Centralized vs. Mean-Field Shared Decentralized Information Structure] A corollary of our analysis is that under the centralized information structure, an optimal policy for the McKean-Vlasov problem can be implemented also by interacting particles who have access to only local state information and the mean-field (empirical) measure; see (\ref{SymmetricMcKeanSoln}). Therefore, for large interacting particle models of the type considered, decentralized information structure, as long as there is access to mean-field distribution, leads to a performance which is as well as one attained under  centralized information. Of course, this observation is not new (as reviewed extensively in \cite{sanjarisaldiy2024optimality}), but it is important to state in this approximation context where both the sampling time and the number of agents appear as approximation parameters.
\end{remark}

We end this section by relating the approximation results developed in this paper with existing finite-population approximation results for mean-field control problems. The convergence of finite-agent value functions toward the limiting McKean--Vlasov value function under compactness and weak convergence methods was established by Lacker~\cite{lacker2017limit}. Quantitative convergence rates, including $1/N$-type estimates under stronger regularity assumptions, have also been obtained recently through PDE and master-equation techniques under more restrictive model assumptions; see Jackson et al.~\cite{jackson2023rate}. Compared with \cite{lacker2017limit} the conditions on near optimality of Markov (in state and mean-field) policies presented in Theorems \ref{nearOptDiscFiniteHor} and \ref{nearOptDiscDiscount} are weaker as a convexity condition is not imposed. The theorem below follows from combining the Euler--Maruyama approximation results proved here with propagation-of-chaos estimates available under Lipschitz continuity assumptions. 

\begin{theorem}
\label{thm:Nparticlelimits}
Suppose Assumptions \ref{Alipschitz}, \ref{Abounded}, and \ref{Acost} hold. Then, for all $x\in \Rd$ we have
\[
\lim_{N\to\infty}V_T^N(x)=V_T(x).
\]
\end{theorem}

\begin{proof}
By Theorem~\ref{Tconvevaluefunc}, we have
\[
\bigl|V_T^{h}(x)-V_T(x)\bigr|<\varepsilon/3
\]
for all sufficiently small $h$. Also, \cite[Theorem~4]{sanjarisaldiy2024optimality} yields that, for any fixed $h$, there exists
$N_h\in\mathbb N$ such that
\[
\bigl|V_T^{N,h}(x)-V_T^{h}(x)\bigr|<\varepsilon/3
\qquad\text{for all }N\ge N_h.
\]
Finally, by the uniform Euler--Maruyama estimate Theorem~\ref{TNconvevaluefunc}, it follows from  that
\[
\sup_{N\ge1}\bigl|V_T^{N}(x)-V_T^{N,h}(x)\bigr|\to 0
\qquad\text{as }h\downarrow0.
\]
Hence, by choosing $h$ sufficiently small (if necessary), we may also ensure that
\[
\bigl|V_T^{N}(x)-V_T^{N,h}(x)\bigr|<\varepsilon/3
\,.
\]
Therefore, for all $h$ sufficiently small and $N\geq N_h > 0$, we deduce that
\[
\bigl|V_T^N(x)-V_T(x)\bigr|
\le
\bigl|V_T^N(x)-V_T^{N,h}(x)\bigr|
+
\bigl|V_T^{N,h}(x)-V_T^{h}(x)\bigr|
+
\bigl|V_T^{h}(x)-V_T(x)\bigr|
<\varepsilon.
\]
Since $\varepsilon>0$ is arbitrary, the claim follows.
\end{proof}

By the similar argument as in Theorem~\ref{thm:Nparticlelimits}, we have the following convergence result for the discounted cost case.  

\begin{theorem}
\label{thm:Nparticlelimitsdiscounted}
Suppose Assumptions \ref{Alipschitz}, \ref{Abounded}, and \ref{Acost} hold. Then, for all $x\in \Rd$ we have
\[
\lim_{N\to\infty}V_{\alpha}^N(x)=V_{\alpha}(x).
\]
\end{theorem}

%%%%%%%%%%%%%%%%%%%%
\section{Finite Model Approximations: Near Optimality of Solutions under Weak Feller Regularity of the Discrete-Time Approximate Model}\label{FiniteModApp}

In the previous section, we established the similarity in the behaviour of finite-agent and discrete-time approximations and the McKean-Vlasov limits. Having access to the exact mean-field measure may be very demanding and it would be desirable to be able to approximate this measure.

In this section, we show that one can obtain a near optimal policy for the McKean-Vlasov problem via a finite model approximation obtained through the quantization of the local state and mean-field distribution. Now that we have shown the near optimality of discrete-time approximations, if we can show weak Feller regularity of the discrete-time model, building on \cite[Section 7.2]{sanjarisaldiy2024optimality} and the approximation results in \cite{SaYuLi15c}, we can obtain a finite space Markov Decision Process whose solution will be near optimal for the controlled diffusion. We note that a similar program has been pursued in \cite{motte2022mean} under slightly stronger regularity assumptions, beyond the weak Feller condition. 

For the analysis in this section, we further impose the following non-degeneracy condition. 
\begin{assumption}\label{Anondegeracy}
For all $z=(z_{1},\dotsc,z_{d})\transp\in\RR^{d}$ there exists $\hat{C}_1 >0$, such that  
\begin{equation*}
\sum_{i,j=1}^{d} a^{ij}(x)z_{i}z_{j}
\,\ge\,\hat{C}_{1} \abs{z}^{2} \qquad\forall\, x\in \RR^{d}\,,
\end{equation*}
where $a\df \frac{1}{2}\upsigma \upsigma\transp$\,.
\end{assumption}

Under this additional assumption, all conditions of \cite[Assumption 1]{sanjarisaldiy2024optimality} are satisfied, except for compactness of the state space, which may be replaced by a suitable tightness condition on the sequence ${\cal L}(X_t)$.
%, and that the law of $X_{nh}$, conditioned on $X_{(n-1)h},U_{(n-1)h}$ is non-atomic. 

We next verify that the discrete-time approximation satisfies the regularity properties required for the analysis that follows.
\begin{lemma}\label{regMcV}
Consider the discrete-time model with dynamics (\ref{Eeuler_maruyama}) and cost given in (\ref{Ediscrete_cost}). Then, under Assumptions \ref{Alipschitz}, \ref{Abounded}, \ref{Acost}, and \ref{Anondegeracy}, we have
\begin{itemize}
\item[(i)] The map: 
\[ \Rd \times {\cal P}(\Rd) \times \mathbb{U} \ni (x,\mu,\zeta) \mapsto \mathbb{E}[g(X^h_{k+1}) | X^h_{k}=x, {\cal L}(X^h_k)=\mu,U^h_{k}=\zeta] \in \mathbb{R},\]
is continuous for every bounded continuous $g: \Rd \to \RR$.
\item[(ii)] The cost function $c^h(x,\mu,u):= c(x, \mu, u) h: \Rd \times {\cal P}(\Rd) \times \mathbb{U}  \to \mathbb{R}$ is continuous and bounded.
\item[(iii)] For all $k \geq 0$, ${\cal L}(X^h_k)$ admits a density which is positive everywhere.
\end{itemize}
\end{lemma}

By Theorem \ref{TOptiDiscMcKV1A}(i), Condition (i) above holds, (iii) holds due to the conditions on the diffusion matrix under non-degeneracy, and (ii) holds by hypothesis.

As earlier in the proof of Theorem \ref{TOptiDiscMcKV1A}, one can define a controlled Markov model where $\mu_t = {\cal L}(X_t)$ is the state and the joint measure $$\Theta_t = \{v \in {\cal P}(\Rd \times \mathbb{U}): v(dx\times \mathbb{U})=\mu_t(dx) \}$$ serving as the {\it action} (endowed with the Young topology  \cite{balder1997consequences} \cite{bauerle2010optimal,bauerle2018optimal}\cite{florescu2012young}\cite{yuksel2023borkar}). In particular, $\Theta_t$ defines a control policy for $\mu_t$ almost every $x \in \Rd$ as a relaxed control policy. Then, $(\mu_t, \Theta_t)$ forms a controlled Markov chain and an optimal $\Theta_t$ is uniquely defined by $\mu_t$, leading to an optimal control. Accordingly, under Theorem \ref{regMcV}, one can quantize $\mu_t$ and $\Theta_t$ to obtain a finite model approximation, parametrized by a granularity level $n$, whose solution $\hat U^{h,n}$ is near optimal (i.e., asymptotically optimal as $n \to \infty$) for the continuous space model by \cite[Section 7.2]{sanjarisaldiy2024optimality}. Therefore, near optimality of a finite model approximation for the controlled McKean-Vlasov diffusion problem follows.
%%%%%%%%%%%%%%%%
The above discussion yields the following asymptotic optimality result.
\begin{theorem}
\label{thm:finite_state_near_discrete}
Suppose Assumptions \ref{Alipschitz}, \ref{Abounded}, and \ref{Acost} hold. Let $\hat U^{h,n}$
be an optimal policy for the finite-state approximation of the
discrete-time McKean--Vlasov finite-horizon/discounted control problem (by a slight abuse of notation, we use the same notation for both quantities) with discretization
parameters \(h>0\).
Then, for \(x\in\mathbb R^d\), as $n\to \infty$
\begin{itemize}
\item Finite horizon: 
$\bigl|
\cJ_{T,h}^{\hat U^{h,n}}(x)-V_T^h(x)
\bigr|
\to 0\,.$
\item Discounted infinite horizon: $\bigl|
\cJ_{\alpha,h}^{\hat U^{h,n}}(x)-V_{\alpha}^h(x)
\bigr|
\to 0\,.$
\end{itemize}
\end{theorem}

Combining the finite-state approximation above with the discrete-time approximation result
proved earlier Theorem~\ref{TNearFinite}, we obtain near-optimality of the finite-state optimal policy for the
continuous-time McKean--Vlasov problem.
\begin{theorem}
\label{thm:finite_state_near_optimality}
Suppose Assumptions \ref{Alipschitz}, \ref{Abounded}, and \ref{Acost} hold. Let
\(
\hat U^{h,n}
\)
be an optimal policy for the finite-state approximation of the
discrete-time McKean--Vlasov finite-horizon/discounted control problem (by a slight abuse of notation, we use the same notation for both quantities) with discretization
parameters \(h>0\). Let $\tilde U^{h,n}$
be the corresponding piecewise-constant extension. Then,
for every initial condition \(x\in\mathbb R^d\), as $h\to 0$ and $n\to \infty$
\begin{itemize}
\item Finite horizon: 
$\bigl|
\cJ_T^{\tilde U^{h,n}}(x)-V_T(x)
\bigr|
\to 0\,.$
\item Discounted infinite horizon: $\bigl|
\cJ_{\alpha}^{\tilde U^{h,n}}(x)-V_{\alpha}(x)
\bigr|
\to 0\,.$
\end{itemize}
\end{theorem}

We now combine the above approximation results with the mean-field approximation results established earlier for the \(N\)-agent system Theorem~\ref{nearOptDiscFiniteHor}, Theorem~\ref{nearOptDiscDiscount}, we deduce the following.

\begin{theorem}\label{nearOptDiscDiscountFiniteAppr}
Suppose Assumptions \ref{Alipschitz}, \ref{Abounded}, \ref{Acost}, and \ref{Anondegeracy} hold. 
\begin{itemize}
\item (Finite Horizon)
Let $\hat{U}^{h,n}$ be a near optimal finite horizon optimal policy for the discrete-time McKean-Vlasov control problem obtained via finite model approximation in the above, and let ${\bf {\hat U}}^{N,h,n} = (\hat{U}^{h,n},\dots,\hat{U}^{h,n})$ (that is, the solution adopted by each agent). Then for any $\epsilon >0$, there exist $\bar{n} \in \mathbb{N}$, $\hat{h}_0>0$ and $\hat{N}_0\in\mathbb{N}$ such that for all $n \geq \bar{n}$, $0 < h < \hat{h}_0$ and all $N\geq \hat{N}_0$, %we have
\[
\bigl|\cJ^N_{T}(x,{\bf {\hat U}}^{N,h,n}) - V_T^N(x)\bigr| \;<\; \epsilon.
\]
\item (Discounted Infinite Horizon)
Let $\hat{U}^{h,n}$ be a near optimal discounted optimal policy for the discrete-time McKean-Vlasov control problem obtained via finite model approximation in the above, and let ${\bf {\hat U}}^{N,h,n} = (\hat{U}^{h,n},\dots,\hat{U}^{h,n})$ (that is, the solution adopted by each agent). Then for any $\epsilon >0$, there exist $\bar{n} \in \mathbb{N}$, $\hat{h}_0>0$ and $\hat{N}_0\in\mathbb{N}$ such that for all $n \geq \bar{n}$, $0 < h < \hat{h}_0$ and all $N\geq \hat{N}_0$, we have
\[
\bigl|\cJ^N_{\alpha}(x,{\bf {\hat U}}^{N,h,n}) - V_{\alpha}^N(x)\bigr| \;<\; \epsilon.
\]

\end{itemize}

\end{theorem}

Given the prior results, the proof is immediate by a triangle inequality argument: As $\hat{U}^{h,n}$ approximates the optimal policy $\hat{U}^{h}$ for large enough $n$, it suffices to take $N$, the number of particles large enough and the time-discretization parameter $h$ small enough so that the performance attained by these large interacting agents are approximated by their respective McKean-Vlasov dynamics under $\hat{U}^{h,n}$ and $\hat{U}^{h}$, respectively. %The above also applies for the finite horizon problem. 

\section{Conclusion}\label{Sconclusion}

In this paper, for the McKean-Vlasov diffusion model, we study the continuity of the expected cost with respect to the control policy, establish the existence of optimal control policies, and show that discrete-time approximations can be used to construct near-optimal policies with arbitrarily small performance loss. The resulting discrete-time model admits efficient numerical schemes with rigorous performance guarantees under mild regularity conditions.

We also show that such an approximate solution is near optimal for $N$-particle interacting systems for large $N$. We thus developed a general approximation framework in both time (via time discretization), space (via measure space quantization under weak Feller regularity for the discrete-time model), and the number of agents (via near optimality as $N$ grows under complementary conditions to those in the literature).

\section{Acknowledgement}
This research of the first author was partially supported by a Start-up Grant IISERB/ R\&D/2024-25/154 and Prime Minister Early Career Research Grant ANRF/ECRG/2024/001658/ PMS. The research of the second author was partially supported by the Natural Sciences and Engineering Research Council of Canada (NSERC)
%%%%%%%%%%%%%%%%%%%%%%

\bibliographystyle{plain}
\bibliography{Quantization,SerdarBibliography}

\end{document}